\documentclass{amsart}
\usepackage{tikz}
\usepackage{xcolor}
\usepackage{amssymb,latexsym,amsmath,extarrows,mathabx}
\usepackage{graphicx,mathrsfs}
\usepackage{hyperref,url}

\usepackage[shortlabels]{enumitem}

\numberwithin{equation}{section}

\newtheorem{theorem}{Theorem}[section]
\newtheorem{lemma}[theorem]{Lemma}

\newtheorem{proposition}[theorem]{Proposition}
\newtheorem{remark}[theorem]{Remark}

\newtheorem{definition}[theorem]{Definition}
\newtheorem{corollary}[theorem]{Corollary}

\newcommand{\al}{\alpha}
\newcommand{\be}{\beta}
\newcommand{\ga}{\gamma}

\newcommand{\de}{\delta}

\newcommand{\De}{\Delta}
\newcommand{\e}{\varepsilon}

\newcommand{\ka}{\kappa}
\newcommand{\la}{\lambda}

\newcommand{\si}{\sigma}

\newcommand{\eps}{\eta}

\newcommand{\cq}{\mathcal Q}
\newcommand{\cp}{\mathcal P}

\newcommand{\cb}{\mathcal B}
\newcommand{\ce}{\mathcal E}
\newcommand{\cd}{\mathcal D}

\newcommand{\cj}{\mathcal J}

\newcommand{\ZR}{\mathbb{R}}

\newcommand{\ZS}{\mathbb{S}}

\newcommand{\ck}{{\mathcal K}}
\newcommand{\cT}{{\mathcal T}}

\newcommand{\cl}{{\mathcal L}}

\newcommand{\ang}{\measuredangle}

\author{Hong Wang}

\address{Courant institute of mathematical sciences, New York University}

\email{hw3639@nyu.edu}

\author{Shukun Wu}

\address{Department of Mathematics, Indiana University Bloomington}

\email{shukwu@iu.edu }

\title{Two-ends Furstenberg estimates in the plane}

\begin{document}

\begin{abstract}
We prove two-ends Furstenberg estimates in the plane for a Katz--Tao $(\de,t)$-set of lines, for general $t\in[0,2]$.
\end{abstract}

\maketitle

\section{Introduction}

In this paper, we extend the two-ends Furstenberg inequality established in \cite{Wang-Wu} to a broader class of lines.
We begin with a standard definition.

\begin{definition}[Shading]
\label{shading-def}
Let $L$ be a family of lines in $\ZR^2$ and let $\de\in(0,1)$.
A {\bf shading} $Y:L\to B^2(0,1)$ is an assignment such that $Y(\ell)\subset N_\de(\ell)\cap B^2(0,1)$ is a union of $\de$-balls in $\ZR^2$ for all $\ell\in L$.
We write $(L,Y)_\de$ to emphasize the dependence on $\de$. 
A similar definition applies to a family of $\de$-tubes $\cT$.
\end{definition}

Regarding the distribution of a shading $Y(\ell)$, the next definition gives a quantitative ``two-ends" condition, which is a fairly weak non-concentration condition.

\begin{definition}[Two-ends]
\label{two-ends-def}
Let $\de\in(0,1)$ and let $(L,Y)_\delta$ be a set of lines and shading.
Let $0<\e_2<\e_1<1$.
We say $Y$ is {\bf $(\e_1 ,\e_2, C)$-two-ends} if for all $\ell\in L$ and all $\de\times\de^{\e_1}$-tubes $J\subset N_\de(\ell)$, 
\begin{equation}
    |Y(\ell)\cap J|\leq C\de^{\e_2} |Y(\ell)|.
\end{equation}
When the constant $C$ is not important in the context, we say $Y$ is {\bf $(\e_1 ,\e_2)$-two-ends}, or simply {\bf two-ends}. 
A similar definition applies to a single shading $Y(\ell)$.
\end{definition}

\smallskip

Now let us recall \cite[Theorem 2.1]{Wang-Wu}:
\begin{theorem}
\label{two-ends-furstenberg}
Let $\de\in(0,1)$.
Let $(L,Y)_\de$ be a set of directional $\de$-separated lines in $\ZR^2$ with an $(\e_1, \e_2)$-two-ends shading such that $|Y(\ell)|\geq\la\de$ for all $\ell\in L$. 
Then for all $\e>0$,
\begin{equation}
\label{eq:thm2.1}
    \Big|\bigcup_{\ell\in L}Y(\ell)\Big|\geq c_\e\de^{\e}\de^{\e_1/2} \la^{1/2}\sum_{\ell\in L}|Y(\ell)|.
\end{equation}
\end{theorem}

As discussed in \cite[Section 0.3]{Wang-Wu}, \eqref{eq:thm2.1} matches the numerology of the Furstenberg set inequality established in \cite[Theorem 4.1]{ren2023furstenberg} (see Theorem \ref{furstenberg-thm}).
The assumption that $L$ is a set of directional $\de$-separated lines can easily be replaced by the superficially stronger assumption that $L$ is a Katz--Tao $(\de,1)$-set of lines.
The main result of this paper is a generalization of Theorem \ref{two-ends-furstenberg} to a set of lines and two-ends shading $(L,Y)_\de$, where $L$ is a Katz--Tao $(\de,t)$-set of lines for general $t\in[0,2]$ (see Definition \ref{point-line-def} for the definition of a Katz--Tao $(\de,t)$-set of lines).
We begin with a necessary definition, which is essentially the Katz--Tao constant introduced in Definition \ref{KT-set}.

\begin{definition}
\label{gamma-Y-def}
Let $\de\in(0,1]$, and let $(L,Y)_\de$ be a set of lines and shading.
Let $t\in[0,1]$.
For each $\ell\in L$, define 
\begin{equation}
\label{gamma-y}
    \ga_{Y,t}(\ell):=\sup_{ r\in[\de,1],\, x\in\ell}(\de/r)^{t}(\de^{-2}|B(x,r)\cap Y(\ell)|),
\end{equation}
and define $\ga_{Y,t}:=\sup_{\ell\in L}\ga_{Y,t}(\ell)$.
\end{definition}

\begin{theorem}
\label{two-ends-furstenberg-general-intro}
Let $(L,Y)_\de$ be a set of $\de$-separated lines in $\ZR^2$ with a $(\e_1, \e_2)$-two-ends shading  such that $|Y(\ell)|\geq\la\de$ for all $\ell\in L$.
Let $t\in(0,2)$, and $t^\ast=\min\{t, 2-t\}$. 
Suppose the set of lines $L$ is a Katz--Tao $(\de, t)$-set.
Then for all $\e>0$,
\begin{equation}\label{eq-intro}
    \Big|\bigcup_{\ell\in L}Y(\ell)\Big|\geq c_\e\de^{\e}\de^{t\e_1/2} \la^{1/2}\de^{(t-1)/2}\ga_{Y,t^\ast}^{-1/2}\sum_{\ell\in L}|Y(\ell)|.
\end{equation}
\end{theorem}

Compared to \eqref{eq:thm2.1}, an additional factor $\ga_{Y,t^\ast}^{-1/2}$ appears in \eqref{eq-intro}, which always equals $1$ when $t=1$.
When $t\not=1$, $\ga_{Y,t^\ast}^{-1/2}$ is generally smaller than $1$, and, unfortunately, necessary.
See Remark \ref{sharp-example} for some examples.

\smallskip

As a corollary, we have
\begin{corollary}
Let $(L,Y)_\de$ be a set of $\de$-separated lines in $\ZR^2$ with a $(\e_1, \e_2)$-two-ends shading.
Let $t\in(0,2)$.
Suppose that $L$ is a Katz--Tao $(\de, t)$-set.
Suppose also that for each $\ell\in L$,  $|Y(\ell)|\geq\la\de$, and $Y(\ell)$ is a Katz--Tao $(\de, \min\{t, 2-t\})$-set.
Then for all $\e>0$,
\begin{equation}\label{eq-cor}
    \Big|\bigcup_{\ell\in L}Y(\ell)\Big|\geq c_\e\de^{\e}\de^{t\e_1/2} \la^{1/2}\de^{(t-1)/2}\sum_{\ell\in L}|Y(\ell)|.
\end{equation}
\end{corollary}

\bigskip

\noindent {\bf Notation:} Throughout the paper, we use $\# E$ to denote the cardinality of a finite set.
We denote by $|E|$ the Lebesgue measure of a set $E\subset \ZR^n$.
If $\ce$ is a family of sets in $\ZR^n$, we use $\cup_\ce$ to denote $\cup_{E\in\ce}E$.
For $A,B\geq 0$, we use $A\lesssim B$ to mean $A\leq CB$ for an absolute (big) constant $C$, and use $A\sim B$ to mean $A\lesssim B$ and $B\lesssim A$.
For a given $\de<1$, we use $ A \lessapprox B$ to denote $A\leq c_\e\de^{-\e} B$ for all $\e>0$.

\bigskip

\section{Preliminaries}\label{section: incidence geometry}

We recall some basic definitions and tools in this section.

\begin{definition}
For two finite sets $E,F$, we say $E$ is a  $\gtrsim c$-refinement of $F$, if $E\subset F$ and $\#E\gtrsim c\#F$; we say $E$ is a $\gtrapprox c$-refinement of $F$, or simply a {\bf refinement} of $F$, if $E\subset F$ and  $\#E\gtrapprox c\#F$.
A similar definition applies when $E,F$ are two finite unions of $\de$-balls.
\end{definition}

\begin{definition}
\label{E-rho}
Let $\delta\in(0,1)$ be a small number.
Let $E$ be a finite union of $\de$-balls in $\ZR^n$. 
For another $\rho\geq\de$, let
\begin{equation}
\label{rho-covering}
    |E|_\rho=\min\{\#\cd_\rho:\cd_\rho\text{ is a covering of $E$ by $\rho$-balls}\}.
\end{equation}
Let $\cd_\rho(E)$ be a family of $\rho$-balls that attains the minimum in \eqref{rho-covering}.
Denote by 
\begin{equation}
\label{E-sub-rho}
    (E)_\rho=\cup_{\cd_\rho(E)}.
\end{equation}
The same definitions apply to $\cp$ when it is a finite set of disjoint $\de$-balls in $\ZR^n$ by considering $\cup_{\cp}$.
\end{definition}

\smallskip

To capture the distribution of a fractal set, we need the following two notations.

\begin{definition}[$(\de,s,C)$-set]
\label{delta-s}
Let $\delta\in(0,1)$ be a small number.
For $s\in(0,n]$, a non-empty set $E\subset \ZR^n$ is called a {\bf $(\de,s,C)$-set} (or simply a {\bf $(\de,s)$-set} if $C$ is not important in the context) if 
\begin{equation}
    |E\cap B(x,r)|_\de\leq Cr^s|E|_\de, \hspace{.3cm}\forall x \in\ZR^n, \,r\in[\de,1].
\end{equation}

\end{definition}

\begin{definition}[Katz--Tao $(\de,s,C)$-set]
\label{KT-set}
Let $\delta\in(0,1)$ be a small number.
For $s\in(0,n]$, a finite set $E\subset \ZR^n$ is called a {\bf Katz--Tao $(\de,s,C)$-set} (or simply a {\bf Katz--Tao $(\de,s)$-set} if $C$ is not important in the context) if 
\begin{equation}
    \#(E\cap B(x,r))\leq C(r/\de)^s, \hspace{.3cm}\forall x \in\ZR^n, \,r\in[\de,1].
\end{equation}
\end{definition}

By a standard point-line duality (see, e.g., \cite[Definition 1.13]{Wang-Wu}), we have
\begin{definition}
\label{point-line-def}
\rm
A set of lines in $\ZR^2$ is a Katz--Tao $(\de,t)$-set (respectively, a $(\de,t)$-set) if its dual point set in $\ZR^2$ is a Katz--Tao $(\de,t)$-set (respectively, a $(\de,t)$-set).
\end{definition}

\smallskip

The next two lemmas are standard. 
See \cite{Wang-Wu}.

\begin{lemma}
\label{refinement-delta-s-set}
Let $\delta\in(0,1)$ and let $C_1, C_2\geq1$.
If a union of $\de$-balls $E$ is a $(\de,s, C_1)$-set and $E'$ is a $\geq C_2^{-1}$-refinement of $E$, then $E'$ is a $(\de,s, C_1C_2)$-set. 
\end{lemma}

\begin{lemma}
\label{Katz--Tao-set-lem}
Let $0<\de<\rho<1$. 
Suppose $E$ is a Katz--Tao $(\de,s,C)$-set, where $1\leq C\lessapprox1$.
Then there exists $E'\subset E$ such that $\rho^s\#E'\gtrapprox\de^s\#E$, and $E'$ is a Katz--Tao  $(\rho,s,C')$-set for some $C'\lesssim1$.
\end{lemma}

\smallskip

The following is a generalization of Definition \ref{delta-s}, quantifying a similar non-concentration condition, though at coarser scales.

\begin{definition}[$(\de,s,C; \De)$-set]
\label{delta-s-Delta}
Let $\de\in(0,1)$ and let $\De\in[\de,1]$.
We say a union of $\de$-balls $E$ is a {\bf $(\de,s,C; \De)$-set}, if 
\begin{equation}
    |E\cap B(x,r)|_\de\leq Cr^s|E|_\de, \hspace{.3cm}\forall x \in\ZR^n, \,r\in[\De,1].
\end{equation}
\end{definition}
Note that a $(\de,s,C)$-set is a $(\de,s,C;\De)$-set for all $\De\in[\de,1]$.
When $\De=\de$, a $(\de,s,C;\De)$-set is just a $(\de,s,C)$-set.
Also, at the coarser resolution $\De$, after uniformization (see Lemma \ref{uniformization}), a $(\de,s,C;\De)$-set is a $(\De,s,C)$-set.
Moreover, similar to Lemma \ref{refinement-delta-s-set}, we have
\begin{lemma}
\label{refinement-delta-s-Delta-set}
Let $0<\de\leq\De\leq1$ and let $C_1, C_2\geq1$.
If a union of $\de$-balls $E$ is a $(\de,s, C_1;\De)$-set and $E'$ is a $\gtrsim C_2^{-1}$-refinement of $E$, then $E'$ is a $(\de,s, C_1C_2;\De)$-set. 
\end{lemma}

\smallskip

Next, we recall the process of uniformization.
Compared to that of an arbitrary set, the distribution of a uniform set is easier to describe.
As we will see in Definition \ref{branching-fcn-def}, this distribution can be captured precisely by a Lipschitz function known as the branching function of the uniform set.

\begin{definition}
\label{m-uniform}
Let $\delta\in(0,1)$ be a small number.
Let $M=|\log\de|$ and let $\rho_j = M^{-j}$, $j=1,\dots, \lceil\log_M\de^{-1}\rceil$.
Given $E$ a union of $\de$-balls in $\ZR^n$, we say $E$ is {\bf uniform  with error $C$}  if $|E\cap D_{\rho_j}|$ are  the same up to a multiple of $C$ for all $D_{\rho_j}\subset(E)_{\rho_j}$. 
When $C$ is not important in the context, we say $E$ is {\bf uniform}. 
The same definitions apply to $\cp$ when it is a finite set of disjoint $\de$-balls in $\ZR^n$ by considering $\cup_{\cp}$.

\end{definition}

The following lemma shows that for an arbitrary set, we can always find a refinement that is also uniform. 
See \cite[Lemma 1.11]{Wang-Wu} for details.

\begin{lemma}[Uniformization]
\label{uniformization}
Let $\delta\in(0,1)$ be a small number.
If $E\subset\ZR^n$ is a union of dyadic $\de$-balls, then there is a $\gtrsim (\log|\log\de|)^{-\frac{|\log\de|}{\log|\log\de|}}$-refinement $E'$ of $E$ such that $E'$ is uniform.
\end{lemma}

\begin{definition}[Branching function]
\label{branching-fcn-def}
Let $\delta\in(0,1)$ and let $M=|\log\de|$.
Let $\rho_j =M^{-j},\, j=1, \dots,  \lceil \log_M\de^{-1}\rceil=: N $.
Let $E\subset\ZR^n$ be a union of $\de$-balls that is uniform.
Define a {\bf branching function} $\be_E:[0,N]\to[0,n]$ as
\begin{equation}
    \be_E(j)=\frac{\log(|E|_{\rho_j})}{ |\log \delta|}, \hspace{.3cm}0\leq j\leq N
\end{equation}
and interpolate linearly between $\beta_E( j)$ and $\beta_E(j+1)$. 
\end{definition}

Given a family of uniform sets, the next lemma shows that this family has a refinement whose sets have branching functions that are essentially the same.
For its proof, we refer to \cite{Wang-Wu}.

\begin{lemma}
\label{uniform-sets-branching-lem}
Let $\delta\in(0,1)$, $M=|\log\de|$, and  $N=\lceil \log_M \delta^{-1} \rceil$.
Suppose $\ce$ is a finite family of uniform sets of $\de$-balls in the unit ball.
Then there is subset $\ce'\subset\ce$ with $\# \ce'\gtrsim (n \log M)^{-N}\#\ce $ 
and a {\bf uniform branching function $\be_{\ce'}$ of $\ce'$} such that for any $E\in\ce'$, $|\be_{E}-\be_{\ce'}|\leq  |\log \delta|^{-1}$.    
In particular, $\#\ce'\gtrapprox\#\ce$.

\end{lemma}

\smallskip

After rescaling, the branching function of a uniform set becomes a 1-Lipschitz function.
To study such functions, Shmerkin \cite{shmerkin2023non} introduced a powerful tool known as multi-scale decomposition, which has since been developed in various forms.
The version we need is the following, whose proof can be found in \cite{Demeter-Wang}.

\begin{lemma}
\label{multiscale-prop}
Let $\eta>0$ be a small number and let $\eta_0=\eta_0(\eta)=\eta^{2\eta^{-1}}$. 
Then for a non-decreasing 1-Lipschitz function $f:[0,1]\to[0,1]$, there exists a partition
\begin{equation}
    0=A_1<A_2<\cdots<A_{H+1}=1
\end{equation}
and a sequence 
\begin{equation}
    0\leq s_1< s_2<\cdots< s_{H}\leq 1
\end{equation}
such that for each $1\leq h\leq H$, we have the following:
\begin{equation}
    A_{h+1}-A_h\geq\eta_0\eta^{-1};
\end{equation}
\begin{align}
    f(x)\geq& f(A_h)+s_h(x-A_h)-\eta(A_{h+1}-A_h), \text{ for all }x\in[A_h,A_{h+1}],\\
    &f(A_{h+1})\leq f(A_h)+(s_h+3\eta)(A_{h+1}-A_h);
\end{align}
\begin{equation}
    s_H\geq f(1)-f(0)-\eta.
\end{equation}
\end{lemma}

\medskip

What follows are several definitions related to shadings.

\begin{definition}[$E_{L,Y}$]
\label{E-L-def}
Let $\de\in(0,1)$ and let $(L,Y)_\delta$ be a set of lines and shading.
Define $E_{L,Y}:=\bigcup_{\ell\in L}Y(\ell)$, which can be identified as a union of $\de$-balls. If the shading $Y$ is apparent from the context, we will use $E_{L}$ to denote $E_{L,Y}$.
Moreover, for each $x\in E_{L,Y}$, define 
\begin{equation}
    L_Y(x)=\{\ell\in L:x\in Y(\ell)\}.
\end{equation}
If the shading $Y$ is apparent from the context, we use $L(x)$ to denote $L_Y(x)$.
\end{definition}

\begin{definition}
\label{lambda-dense-def}
Let $\de\in(0,1)$ and let $(L,Y)_\delta$ be a set of lines and shading.
We say $Y$ is {\bf $\la$-dense}, if $|Y(\ell)|\geq \la |N_\de(\ell)|$.
\end{definition}

\begin{definition}
Let $\de\in(0,1)$ and let $(L,Y)_\delta$ be a family of lines and shadings.
We say $(L',Y')_\de$ is a {\bf refinement} of $(L,Y)_\delta$, if $L'\subset L$, $Y'(\ell)\subset Y(\ell)$ for all $\ell\in L'$, and if
\begin{equation}
    \sum_{\ell'\in L'}|Y'(\ell')|\gtrapprox\sum_{\ell\in L}|Y(\ell)|.
\end{equation}
Similarly, for a family of $\de$-tubes and shadings $(\cT,Y)$, we say $(\cT',Y')$ is a {\bf refinement} of $(\cT,Y)$, if $\cT'\subset \cT$, $Y'(T)\subset Y(T)$ for all $T\in \cT'$, and if $\sum_{T'\in \cT'}|Y'(T')|\gtrapprox\sum_{T\in \cT}|Y(T)|$.
\end{definition}

\smallskip

Although a shading $Y(\ell)$ is a union of $2$-dimensional balls, it is indeed a one-dimensional object.
Thus, aligning with \eqref{E-sub-rho}, we make the following definition.

\begin{definition}
\label{tube-segment}
Let $0<\delta<r < 1$. 
Let $\ell$ be a line, and let $Y(\ell)$ be a shading by $\de$-balls. 
We define $(Y(\ell))_r$ as follows:
Let $\cj(\ell)$ be a minimal covering of $Y(\ell)$ by $\de\times r$-tubes contained in $N_\de(\ell)$. 
Now define $(Y(\ell))_r=\cup_{\cj(\ell)}$. 
\end{definition}

\begin{definition}
\label{two-ends-reduction}
Let $v, C>0$, and let $\de\in(0,1)$.
Let $\ell$ be a line and $Y(\ell)$ be a uniform shading by $\de$-balls.
Define $\rho=\rho(\ell;v,C)\in[\de,1]$ as
\begin{equation}
\label{two-ends-reduction-1-section-1}
    \rho:=\min\{r\in[\de,1]:|Y(\ell)|_r< C^{-1}r^{-v}\}.
\end{equation}
Consequently, since $Y(\ell)$ is uniform, for all $r\in[\de,\rho]$ and all $J\subset (Y(\ell))_\rho$, 
\begin{equation}
\label{two-ends-reduction-2-section-1}
    |Y(\ell)\cap J|_r\gtrapprox(r/\rho)^v.
\end{equation}
\end{definition}

Definition \ref{two-ends-reduction} is the standard two-ends reduction on a shading $Y(\ell)$.
Note that, with $\rho=\rho(\ell;v,C)$, the $\rho^{-1}$-dilate of $Y(\ell)\cap J$ is a $(\de/\rho,v, CC')$-set for all $\de\times\rho$-tubes $J\subset (Y(\ell))_\rho$, for some $C'\lessapprox1$.
The next two lemmas were proved in \cite{Wang-Wu}.

\begin{lemma}
\label{delta-1-rho}
Let $\de\in(0,1)$, let $\ell$ be a line, and $Y(\ell)$ be a uniform shading by $\de$-balls.
Let $0<\e_2<\e_1<1$, and let $v<\e_2$.
Suppose $Y(\ell)$ is $(\e_1,\e_2, C)$-two-ends, and let $\rho=\rho(\ell;v,C)$ be the scale given by Definition \ref{two-ends-reduction}.
Then $\rho\geq \de^{\e_1}$.
\end{lemma}

\begin{lemma}
\label{two-ends-shading-lem}
Let $\de\in(0,1)$, let $\ell$ be a line, and $Y(\ell)$ be a shading by $\de$-balls.
If $Y(\ell)$ is $(\e_1,\e_2,C)$-two-ends and $Y'(\ell)$ is a refinement of $Y(\ell)$, then there exists $C'\lessapprox1$ such that $Y'(\ell)$ is $(\e_1,\e_2,CC')$-two-ends.
\end{lemma}

\smallskip

Recall Definition \ref{gamma-Y-def}.
The next two lemmas are about properties of $\ga_Y$.

\begin{lemma}
\label{gamma-y-lemma}
Let $0<\de\leq r\leq1$ and let $\ell$ be a line.
Suppose $Y(\ell)$ is a shading by $\de$-balls and $\tilde Y(\ell)$ is another shading by $r$-balls such that each $r$-ball in $\tilde Y(\ell)$ contains $\gtrsim d$ many $\de$-balls in $Y(\ell)$ for some $d\in[1,r/\de]$.
Then $\ga_{\tilde Y,t}(\ell)\lesssim(\de/r)^{-t}d^{-1}\ga_{Y,t}(\ell)$.
\end{lemma}
\begin{proof}
Let $\rho$ be the scale such that for some ball $B(x,\rho)$,
\begin{equation}
    \ga_{\tilde Y,t}(\ell)\sim (r/\rho)^{t}|B(x,\rho)\cap \tilde Y(\ell)|_r.
\end{equation}
Since each $r$-ball in $\tilde Y(\ell)$ contains $\sim (r/\de)d$ many $\de$-balls in $Y(\ell)$, we have
\begin{equation}
    \ga_{Y,t}(\ell)\gtrsim (\de/\rho)^{t}|B(x,\rho)\cap Y(\ell)|_r\cdot d.
\end{equation}
Therefore,
\begin{equation}
    \ga_{\tilde Y,t}(\ell)\lesssim(\de/r)^{-t}d^{-1}\ga_{Y,t}(\ell). \qedhere
\end{equation}
\end{proof}

\begin{lemma}
\label{multi-scale-lem}
Let $(L,Y)_\de$ be a set of lines and shading such that $\{Y(\ell):\ell\in L\}$ is a set of uniform shadings admitting a uniform branching function (recall Lemma \ref{uniform-sets-branching-lem}).
Let $\eta>0$ and let $\eta_0=\eta_0(\eta)=\eta^{2\eta^{-1}}$. 
Let $t\in[0,1]$ be given.
Then there exist  $r\in(\delta^{1-\eta_0\eta^{-1}},1]$ and $s\in[0,1]$ such that the following is true:

Let $\tilde{Y}(\ell) = (Y(\ell))_r$ for all $\ell\in L$. 
We have
\begin{enumerate}[(a)]
    \item $\ga_{\tilde Y, t}(\ell)\lesssim \ga_{Y,t}(\ell)$. 
    \item \label{it: H} For each $\delta \times r$-tube $J\subset (Y(\ell))_r$, the $r^{-1}$-dilate of $Y(\ell)\cap J$ along $\ell$ is a $(\delta/r, s, (\delta/r)^{-9\eta})$-set, and
    $\log_{1/\de}\big(\frac{|Y(\ell)|_{\de}}{|Y(\ell)|_{r}}\big)\leq (s+9\eta) \log_{1/\de}(r/\de)$.
\end{enumerate}
\end{lemma}

\begin{proof}

Let $f$ be the uniform branching function of $\{Y(\ell):\ell\in L\}$.
Apply Lemma \ref{multiscale-prop} to $f$ to obtain a set of scales 
\begin{equation}
    0=A_1<A_2<\cdots<A_{H+1}=1
\end{equation}
and a sequence 
\begin{equation}
    0\leq s_1< s_2<\cdots< s_{H}\leq 1
\end{equation}
such that the following is true
\begin{equation}
    A_{h+1}-A_h\geq\eta_0\eta^{-1};
\end{equation}
\begin{align}
\label{lower-bound}
    f(x)\geq& f(A_h)+s_h(x-A_h)-\eta(A_{h+1}-A_h), \text{ for all }x\in[A_h,A_{h+1}],\\ \label{upper-bound}
    &f(A_{h+1})\leq f(A_h)+(s_h+3\eta)(A_{h+1}-A_h).
\end{align}

We consider two distinct cases.

\smallskip

Suppose $ s_H \geq t+\eta$.
Take $r=\delta^{A_H}$ and $s_H'=s_H$ so that each $\de\times r$-segment in $(Y(\ell))_r$ contains $\gtrsim(r/\de)^{t}$ many $\de$-balls in $Y(\ell)$.
Thus, by Lemma \ref{gamma-y-lemma}, we have $\gamma_{\tilde{Y}, t}(\ell) \leq \gamma_{Y,t}(\ell)$. 

Suppose $s_H < t+\eta$.
Take $r\geq \delta^{A_H}$ be the largest scale such that for all $B(x,r)\cap Y(\ell)\not=\varnothing$, we have
\begin{equation}
    |B(x,r)\cap (Y(\ell))_{\delta^{A_H}}|_{\delta^{A_H}} \geq (r/\delta^{A_H})^{t}.
\end{equation}
The choice of $r$ implies $\gamma_{\tilde{Y}, t}(\ell)=1$.
Thus, item $(a)$ follows from the simple fact that $1 \leq \gamma_{Y, t}(\ell)$.
Take $s_H'$ be such that $\frac{|Y(\ell)|_{\delta}}{|Y(\ell)|_{r}} \sim (r/\delta)^{s_H'}$.
By \eqref{lower-bound}, we know that for all $\delta\times r$-tube $J\subset (Y(\ell))_r$,
\begin{equation}
    |J\cap Y(\ell)|_\de\approx (r/\de)^{s'_H}\gtrsim (\de^{A_H}/\de)^{-\eta}(r/\de^{A_H})^t(\de^{A_H}/\de)^{s_H}.
\end{equation}
Since the sequence $\{s_j\}$ is increasing, for each $\rho\in[\de,r]$, by \eqref{lower-bound} and \eqref{upper-bound}, 
\begin{equation}
     |(Y(\ell))_{\rho}|_\de\lesssim (r/\de)^{3\eta}(\rho/\de)^{s_H}.
\end{equation}
Therefore, for all $\rho\in[\de,r]$,
\begin{equation}
    |J\cap Y(\ell)|_\rho\gtrapprox (r/\de)^{-4\eta}(r/\de^{A_H})^{t}(\de^{A_H}/\rho)^{s_H}.
\end{equation}
On the other hand, by \eqref{upper-bound}, we have
\begin{equation}
    (r/\de)^{s'_H}\lesssim (\de^{A_H}/\de)^{3\eta}(r/\de^{A_H})^t(\de^{A_H}/\de)^{s_H}.
\end{equation}
This shows
\begin{equation}
    |J\cap Y(\ell)|_\rho\gtrapprox (r/\de)^{-9\eta}(r/\de)^{s_H'},
\end{equation}
which is what we want for item $(b)$.
\qedhere

\end{proof} 

\medskip

We end with two standard lemmas, whose proof can be found in \cite{Wang-Wu}.

\begin{lemma}
\label{rich-point-refinement}
Let $\de\in(0,1)$ and let $(L,Y)_\delta$ be a set of lines in $\ZR^n$ and shading.
There exists a $\mu\geq1$, a set $E^\mu\subset E_L$, and a  refinement $(L',Y')_\de$ of $(L,Y)_\de$ so that 
\begin{enumerate}
    \item $Y'(\ell)$ is a refinement of $Y(\ell)$ for all $\ell\in L'$.
    \item $\#L_{Y'}(x)\sim\mu$ for all $x\in E_{L,Y'}$.
    \item $Y'(\ell)= E^\mu\cap N_\de(\ell)$ for all $\ell\in L'$.
    \item $\mu\approx |E_{L,Y'}|^{-1}\sum_{\ell\in L'}|Y(\ell')|$.
\end{enumerate}
\end{lemma}

\begin{lemma}
\label{broad-narrow-lem}
Let $\de\in(0,1)$, and let $(L,Y)_\delta$ be a set of lines in $\ZR^n$ and shading. 
For all $x\in E_L$, there exists a $\rho=\rho(x)\in[10\de,1]$ so that the following is true.

\begin{enumerate}
    \item There exists a refinement  $L'(x)\subset L(x)$ such that $\ang(\ell,\ell')+\de\leq 2\rho$ for any $\ell,\ell'\in L'(x)$.
    \item There are two disjoint subsets $L_1,L_2\subset L'(x)$ of lines such that $\# L_1, \# L_2\gtrapprox \# L'(x)$, and $\rho\geq\ang(\ell_1,\ell_2)\gtrapprox \rho$ for all $\ell_1\in L_1,\ell_2\in L_2$.
\end{enumerate}
\end{lemma}

\bigskip

\section{Proof of the main theorem}

Let us restate Theorem \ref{two-ends-furstenberg-general-intro} below.

\begin{theorem}
\label{two-ends-furstenberg-general}
Let $(L,Y)_\de$ be a set of $\de$-separated lines in $\ZR^2$ with a $(\e_1, \e_2)$-two-ends, $\lambda$-dense shading.
Let $t\in(0,2)$, and $t^\ast=\min\{t, 2-t\}$. 
Suppose the set of lines $L$ is a Katz--Tao $(\de, t)$-set.
Then for any $\e>0$,
\begin{equation}\label{eq: Thm2.1}
    |E_L|\geq c_\e\de^{\e}\de^{t\e_1/2} \la^{1/2}\de^{(t-1)/2}\ga_{Y,t^\ast}^{-1/2}\sum_{\ell\in L}|Y(\ell)|.
\end{equation}
\end{theorem}

\smallskip

For convenience, we introduce the following definition.

\begin{definition}
\label{L[T]-def}
Let $v\in(0,1]$.
For a tube $T\in\ZR^n$ with cross-section radius $v$ and length $1$, let $\cl[T]$ be the family of lines intersecting $T$ that make an angle $\leq v$ with the coreline of $T$.
Given a set of lines $L$, define
\begin{equation}
    L[T]:=L\cap \cl[T].
\end{equation}
\end{definition}

\begin{remark}
\label{sharp-example}

\rm

The necessity of $\ga_{Y,t^\ast}^{-1/2}$ in \eqref{eq: Thm2.1} can be seen by the following examples:

Suppose $r\in(\delta, 1)$.
Let $(L,Y)_r$ be a set of lines and shadings appearing in \cite[Section 1]{Wolff-survey} so that 
\begin{enumerate}
    \item $L$ is a Katz--Tao $(r,t)$-set of lines with $\# L\sim r^{-t}$.
    \item $Y(\ell)$ is a $(r,s)$-set of $r$-balls and $|Y(\ell)|\sim r^{2-s}\sim\la r$.
    \item $|E_L|\approx \la^{1/2}r^{(t-1)/2}\sum_{\ell\in L}|Y(\ell)|\sim\la^{3/2}r^{(1-t)/2}$.
\end{enumerate}

\smallskip

{\bf Case 1: $t\in[0,1]$.}
Assume all lines in $L$ are transversal to the horizontal axis.
Let $(L',Y')_\de$ be the anisotropic $\de/r$-rescaling of $(L,Y)_r$ along the horizontal direction.
Then $L'$ is a Katz--Tao $(\de,t)$-set of lines, and for each $\ell\in L'$, $Y'(\ell)$ is $\la$-dense and is a union of $\de\times r$-segments.
Clearly $Y'$ is two-ends, and $\sum_{\ell\in L'}|Y'(\ell)| \sim\la r^{-t}\de$.
Moreover, we have $\ga_{Y',t^\ast}\sim (r/\de)^{1-t}$.

Now we can calculate
\begin{equation}
    |E_{L'}|\sim (\de/r)|E_L|\sim \la^{1/2}\de^{(t-1)/2}\ga_{Y',t^\ast}^{-1/2}\sum_{\ell\in L'}|Y'(\ell)|,
\end{equation}
which shows the necessity of $\ga_{Y,t^\ast}^{-1/2}$ in \eqref{eq: Thm2.1} when $t\in[0,1]$.

\smallskip

{\bf Case 2: $t\in[1,2]$.}
For each $\ell\in L$, let $L'[\ell]$ be a Katz--Tao $(\de,t)$-set of lines contained in $N_r(\ell)$ with $\# L'[\ell]\sim(r/\de)^{t}$ and $N_r(\ell)\subset \bigcup_{\ell'\in L'[\ell]}N_{\de}(\ell')$.
Let $L'=\bigcup_{\ell\in L}L'[\ell]$, so $L'$ is a Katz--Tao $(\de,t)$-set with $\#L'\sim\de^{-t}$.
For each $\ell'\in L'[\ell]$, let $Y'(\ell')=Y(\ell)\cap N_\de(\ell')$ be a set of $\de\times r$-segments.
Then $Y'$ is $\la$-dense, and 
$\ga_{Y',t^\ast}\sim (r/\de)^{t-1}$.

Now we can calculate
\begin{equation}
    |E_{L'}|\sim |E_L|\sim \la^{1/2}r^{(t-1)/2}\sum_{\ell\in L}|Y(\ell)|\sim \la^{1/2}\de^{(t-1)/2}\ga_{Y',t^\ast}^{-1/2}\sum_{\ell\in L'}|Y'(\ell)|,
\end{equation}
which shows the necessity of $\ga_{Y,t^\ast}^{-1/2}$ in \eqref{eq: Thm2.1} when $t\in[1,2]$.

\end{remark}

\medskip

\subsection{Outline of the proof}

\label{outline-subsection}

For simplicity, we assume $|Y(\ell)|\sim\la\de$ for all $\ell\in L$, and we only aim to heuristically justify the following simplified estimate:
\begin{equation}
\label{large-volume-heurisitcs}
    |E_L|\gtrsim \la^{1/2}\de^{(t-1)/2}\ga_{Y,t^\ast}^{-1/2}\sum_{\ell\in L}|Y(\ell)|
\end{equation}
That is, $E_L$ has {\bf large volume}.
Note that, modulo a few steps of pigeonholing, it is equivalent to the following {\bf small multiplicity} claim: for all $x\in E_L$,
\begin{equation}
\label{small-multiplicity-heurisitcs}
    \# L(x)\lesssim \la^{-1/2}\de^{(1-t)/2}\ga_{Y,t^\ast}^{1/2}.
\end{equation}
We will show that either \eqref{large-volume-heurisitcs} or \eqref{small-multiplicity-heurisitcs} holds, which implies Theorem \ref{two-ends-furstenberg-general}.

\smallskip

We assume that  $Y(\ell)$ admits the same branching function for all $\ell\in L$.
By Lemma \ref{multi-scale-lem}, there is a parameter $s$ and a scale $r$  such that the $r^{-1}$-dilate of the shading $J\cap Y(\ell)$ is a $(\de/r,s)$-set  and $|J\cap Y(\ell)|_\delta \approx (\delta/r)^{-s}$ for all $\de\times r$-tube $J\subset(Y(\ell))_r$ (see Definition~\ref{tube-segment}).

Consider the family $\cT=\{J\subset(Y(\ell))_r:\ell\in L\}$ of $\de\times r$-tubes generated by $(L,Y)_\de$.
We also introduce the new shading $\bar Y(\ell):=(Y(\ell))_r$, which consists of a union of  $\de\times r$-tubes and is $\la(r/\de)^{1-s}$-dense.
Notice that, by Lemma \ref{multi-scale-lem}, we have $\ga_{\bar Y,t^\ast}\lesssim \ga_{Y,t^\ast}$.
Moreover, for each $\de/r\times 1$-tube $\bar T$, the anisotropic $r/\de$-dilate of the configuration $(L[\bar T], \bar Y)_\de$ along the normal direction of  $\bar T$ yields a Katz--Tao $(r,t)$-set of lines in $\ZR^2$ with two-ends, $\la(r/\de)^{1-s}$-dense shadings.
Thus by induction, we can apply \eqref{small-multiplicity-heurisitcs} at scale $r$  to conclude that for all $\de\times r$-tubes $J\in\cT$
\begin{align}
\label{upper-bound-heuristic}
    \#\{\ell\in L:J\subset (Y(\ell))_r\}&\sim\#\{\ell\in L[\bar T]:J\subset (Y(\ell))_r\}\\
    &\lesssim (\la(r/\de)^{1-s})^{-1/2}r^{(1-t)/2}\ga_{\bar Y,t^\ast}^{1/2}.
\end{align}
Here $\bar T$ is the $(\de/r)\times 1$-tube in the same direction as $J$.

\smallskip 

Next, we are going to analyze what happens inside an $r$-ball $B$.
Let $\cT_B$ be the set of $\de\times r$-tubes in $\cT$ that intersect the $r$-ball $B$.
For these $\de\times r$-tubes we define a shading $Y_B$ as follows. 
For each $J\in\cT_B$, we choose a line $\ell\in L$ such that $J\subset (Y(\ell))_r$, and then define  $Y_B(J)=J\cap Y(\ell)$.
As a result, the $r^{-1}$-dilate of the shading $Y_B(J)$ is a $(\de/r,s)$-set.

Suppose $\cT_B$ is uniform, in the sense that its $r^{-1}$–dilate is a uniform set of $\delta/r \times 1$–tubes.
By the standard broad-narrow reduction, we may assume that all points in $ E_{\cT_B, Y_B}$ are broad.
That is, applying Lemma \ref{broad-narrow-lem} to the input $(\cT_B, Y_B)$ yields $\rho\gtrapprox1$.
In particular, this implies that, after point-line duality, the $r^{-1}$-dilate of $(T_B, Y_B)$ obeys the assumption of Theorem \ref{DW-24}.
Thus, we obtain three possible outcomes (see Proposition \ref{prop-induction}):
\begin{enumerate}
    \item {\bf (Small multiplicity):} 
    For each typical $x\in E_{\cT_B, Y_B}$, we have 
    \begin{equation}
    \label{small-multiplicity-heurisitcs2}
        \#\cT_B(x)\lessapprox  (\de/r)^{(s-t)/2}.
    \end{equation}
    \item {\bf (Large volume):}
    For each typical $r$-ball $B$, we have
    \begin{equation}
    \label{large-volume-heurisitcs2}
        |E_L\cap B|\gtrapprox (\delta/r)^{2-(t+3s)/2}|B|.
    \end{equation}
    \item {\bf (Thickened tubes):}
    There exists a scale $\Delta$ much larger than $\delta$ such that for each typical $J\in\cT_B$, we have
    \begin{equation}
    \label{thickened-tube-heurisitcs}
        |E_L\cap N_\Delta(J)|\gtrapprox (\delta/r)^{1-s}|J|.
    \end{equation}
\end{enumerate}

\smallskip

Finally, we show that the main theorem follows in either case.

\smallskip

{\bf Case I.} Suppose \eqref{small-multiplicity-heurisitcs2} is true.
By \eqref{upper-bound-heuristic}, we have
\begin{equation}
    \# L(x)\lesssim  (\la(r/\de)^{1-s})^{-1/2}r^{(1-t)/2}\ga_{\bar Y,t^\ast}^{1/2}\cdot (\de/r)^{(s-t)/2}\lesssim \la^{-1/2}\de^{(1-t)/2}\ga_{\bar Y,t^\ast}^{1/2}.
\end{equation}
This yields \eqref{small-multiplicity-heurisitcs} as $\ga_{\bar Y,t^\ast}\lesssim \ga_{Y,t^\ast}$.

{\bf Case II.} Suppose \eqref{large-volume-heurisitcs2} is true.
Consider a new shading $\tilde Y(\ell)=N_r(Y(\ell))$ for a typical $\ell\in L$.
By assumption, it is $\tilde \lambda$-dense and is two-ends as well, where 
\begin{equation}
\label{lambda-tilde-outline}
    \tilde\lambda:=\la(r/\de)^{1-s}.
\end{equation}

Now we refine $L$ to a Katz--Tao $(r, t)$-set. 
More precisely, apply Lemma \ref{Katz--Tao-set-lem} to obtain a $\gtrapprox(\de/r)^t$-refinement $\tilde L$ of $L$ such that $\tilde L$ is a Katz--Tao $(r,t)$-set.
By induction, apply \eqref{large-volume-heurisitcs} at scale $r$ so that 
\begin{align}
\label{volume-Delta-ball}
    |E_{\tilde L, \tilde Y}| &\gtrsim \tilde\lambda^{1/2}r^{(t-1)/2}\ga_{\tilde Y,t^\ast}^{-1/2}\sum_{\ell\in \tilde L}|\tilde Y(\ell)|\\
    &\gtrsim  (\tilde\lambda/\lambda)^{3/2} \lambda^{1/2}r^{(t-1)/2}\ga_{\tilde Y,t^\ast}^{-1/2}(\delta/r)^{t-1} \sum_{\ell\in  L}| Y(\ell)|.
\end{align}
Finally, combining \eqref{large-volume-heurisitcs2}, \eqref{volume-Delta-ball}, and \eqref{lambda-tilde-outline}, and noting $\ga_{\tilde Y,t^\ast}\lesssim \ga_{Y,t^\ast}$, we have
\begin{align}
    |E_L|  &\sim \min_{B\in\cb_{r,6}}\frac{|E_L\cap B|}{|B|}  \cdot |E_{\tilde L, \tilde Y}| \\
    &\gtrapprox  (\delta/r)^{2-(t+3s)/2} \cdot (\tilde\lambda/\lambda)^{3/2} \la^{1/2}r^{(t-1)/2}\ga_{\tilde Y,t^\ast}^{-1/2}(\delta/r)^{t-1}\sum_{\ell\in  L}| Y(\ell)| \\
    &\gtrsim (\delta/r)^{2-(t+3s)/2}(\tilde\lambda/\lambda)^{3/2} (\de/r)^{(t-1)/2}\cdot \la^{1/2}\de^{(t-1)/2}\ga_{Y,t^\ast}^{-1/2}\sum_{\ell\in L}|Y(\ell)|\\
    &\gtrsim \la^{1/2}\de^{(t-1)/2}\ga_{Y,t^\ast}^{-1/2}\sum_{\ell\in L}|Y(\ell)|. 
\end{align}

{\bf Case III.} Suppose \eqref{thickened-tube-heurisitcs} is true.
Consider the thickened shading $\tilde Y(\ell)=E_L\cap N_{\Delta}(\ell)$.
Then \eqref{thickened-tube-heurisitcs} implies that $\tilde Y(\ell)$ is still $\gtrapprox\lambda$-dense for each typical $\ell\in L$.
That is, passing to the thickened configuration $(L,\tilde Y)_\Delta$ incurs essentially no loss.
Thus, we may apply \eqref{large-volume-heurisitcs} at scale $\Delta$ to conclude the theorem.
We remark that some additional technical work is required to control the Katz--Tao constant $\ga_{\tilde Y,t^\ast}$ for this thickened shading; we defer these details to the formal proof.

\medskip

\subsection{Some incidence results and two-ends reduction}

The proof of Theorem \ref{two-ends-furstenberg-general} relies on several incidence estimates.

\begin{lemma}
\label{furstenberg-upper-range-pre}
Let $0<t<u\leq 1$ and let $\eta\in(0, (u-t)/2 )$.
Let $\cb$ be a collection of $\de$-balls in $\ZR^2$, and for each $B\in\cb$, let $\cT(B)$ be a family of $\de\times1$-tubes intersecting $B$.
Suppose $\cup_\cb$ is a $(\de,2-t,\de^{-\eta})$-set.
Suppose also that for each $B\in\cb$, $\cT(B)$, identified with a set of $delta$-arcs, is a $(\de,u,\de^{-\eta})$-set.
Let $r$ be such that $\#\cT(B)\geq r$ for all $B\in\cb$.
Then for any $\e>0$, 
\begin{equation}
    \#\bigcup_{B\in\cb}\cT(B)\geq c_\e\de^\e \de^{-1}r,
\end{equation}
where $ \#\bigcup_{B\in\cb}\cT(B)$ should be understood as the maximum number of distinct tubes contained in $\bigcup_{B\in\cb}\cT(B)$  (two $\delta\times 1$-tubes $T_1$ and $T_2$ are distinct if $|T_1\cap T_2|\leq |T_1|/2$). 
\end{lemma}

Lemma \ref{furstenberg-upper-range-pre} was stated in \cite[Theorem 4.1]{ren2023furstenberg} with an additional dependence between $\eta$ and $\e$.
However, this dependence can be removed since we are only considering the upper range in the Furstenberg set estimate. 
We refer to \cite[Proposition 4.2]{ren2023furstenberg} for its proof.

\begin{lemma}
\label{furstenberg-upper-range}
Let $0<t<u\leq 1$ and let $\eta\in(0,(u-t)/4)$. 
Let $\De\in[\de^{1-\eta},1]$, and let $\cb$ be a set of $\de$-balls in $\ZR^2$ contained in some   $\De$-ball $B_\De$.
For each $B\in\cb$, let $\cT(B)$ be a set of $\de\times1$-tubes intersecting $B$.
Suppose that the $\De^{-1}$-dilate of $\cb$ is a $((\de/\De),2-t,(\de/\De)^{-\eta})$-set,
and for each $B\in\cb$, $\cT(B)$ is a $(\de,u,(\de/\De)^{-\eta})$-set.
Let $r\geq1$ be such that $\#\cT(B)\geq r$ for all $B\in\cb$.
Then for any $\e>0$, 
\begin{equation}
    \#\bigcup_{B\in\cb}\cT(B)\geq c_\e\de^\e (\De/\de)r.
\end{equation}
\end{lemma}
\begin{proof}
For each $B\in\cb$, let $\cT_{\de/\De}(B)$ be the family of distinct $\de/\De\times1$-tubes containing at least one $\de\times1$-tube in $\cT(B)$.
For each $T_{\de/\De}\in\cT_{\de/\De}(B)$, define $\cT(T_{\de/\De}):=\{T\in\cT(B):T\subset T_{\de/\De}\}$.
By dyadic pigeonholing on $\{\#\cT(T_{\de/\De}):T_{\de/\De}\in\cT_{\de/\De}(B)\}$, there exists a set $\cT_{\de/\De}'(B)\subset\cT_{\de/\De}(B)$ and a number $\mu_{B}$ such that $\#\cT(T_{\de/\De})\sim\mu_B$ for all $T_{\de/\De}\in\cT_{\de/\De}'(B)$ and $\#\cT_{{\de/\De}}'(B)\cdot\mu_B\gtrapprox\#\cT(B)\geq r$.
Since $\cT(B)$ is a $(\de,u, (\de/\De)^{-\eta})$-set, $\cT_{\de/\De}(B)$ is a $(\de/\De, u, (\de/\De)^{-\eta})$-set, yielding that $\cT_{\de/\De}'(B)$ is a $(\de/\De,u,(\de/\De)^{-2\eta})$-set.

For each $\de/\De\times1$-tube $T_{\de/\De}\in\cT_{\de/\De}'(B)$, there is a unique $\de\times \De$-tube $\bar T$ intersecting $B$ that is parallel to $T_{\de/\De}$.
Let $\bar\cT'(B)$ be this set of $\de\times\De$-tubes,
so $\#\bar\cT'(B)=\#\cT_{{\de/\De}}'(B)\gtrapprox r/\mu_B$.
Note that $\bar\cT'(B)$ is indeed the $\De$-dilate of $\cT_{\de/\De}'(B)$. 
In particular, the $\De^{-1}$-dilate of $\bar\cT'(B)$ is a $(\de/\De,u,(\de/\De)^{-2\eta})$-set.
By dyadic pigeonholing on $\{\mu_B:B\in\cb\}$, there exists a refinement $\cb'$ of $\cb$ and a uniform $\mu$ such that $\mu_B\sim\mu$ for all $B\in\cb'$.
Hence, the $\De^{-1}$-dilate of $\cb'$ is a $((\de/\De),2-t,(\de/\De)^{-2\eta})$-set.
Moreover, $\#\cT(T_{\de/\De})\sim\mu$ for all $T_{\de/\De}\in\cT_{\de/\De}'(B)$ and all $B\in\cb'$. 
This shows that there are $\gtrsim \mu$ many $\de\times 1$-tubes in $\cT(B)$ intersecting $\bar T$ for all $\bar T\in\bar\cT'(B)$ and all $B\in\cb'$.
Therefore, $\#\bigcup_{B\in\cb}\cT(B)\gtrsim\mu \#\bigcup_{B\in\cb'}\bar\cT'(B)$.

Apply a $\De$-dilate version of Lemma \ref{furstenberg-upper-range-pre} to the $\de\times\De$-tubes $\{\bar \cT'_B,B\in\cb'\}$ so that
\begin{equation}
    \#\bigcup_{B\in\cb}\cT(B)\gtrsim\mu\cdot  \#\bigcup_{B\in\cb'}\bar\cT'(B)\gtrapprox \mu(\De/\de) (r/\mu)\gtrapprox (\De/\de)r. \qedhere
\end{equation}

\end{proof}

\smallskip

To prove the next result, we need the following Furstenberg set estimate proved in \cite{Orponen-Shmerkin-2} and \cite{ren2023furstenberg}.

\begin{theorem}[\cite{ren2023furstenberg}, Theorem 4.1, Remark 6.8]
\label{furstenberg-thm}
For all $\nu>0$, there exists an $\eta>0$ such that the following is true for all $s\in[0,1]$ and $t\in[0,2]$.

Let $(L,Y)_\de$ be a set of lines and shadings such that $L$ is a $(\de,t,\de^{-\eta})$-set, and $Y(\ell)$ is a $(\de,s,\de^{-\eta})$-set for all $\ell\in L$.
Then
\begin{equation}
    |E_L|\gtrsim \de^{\nu}\de^{2}\de^{-\min\{t+s,\frac{3s+t}{2},1+s\}}.
\end{equation}
\end{theorem}

\begin{lemma}
\label{furstenberg-Katz--Tao-lem}
Let $(\nu,\eta)$ be as in Theorem \ref{furstenberg-thm}.
Let $\la\in(0,1]$.
Let $(L,Y)_\de$ be a set of lines and $\la$-dense shadings such that $L$ is a Katz--Tao $(\de,t)$-set.
Suppose for all $\ell\in L$, $Y(\ell)$ is a $(\de,s,\de^{-\eta})$-set, and $|Y(\ell)|\lesssim \de^{1-s-\eta}|N_\de(\ell)|$.
Then
\begin{equation}
\label{furstenberg-Katz--Tao-esti}
    |E_L|\gtrsim \de^{\nu+O(\eta)}\la^{1/2}\de^{(t-1)/2}\ga_{Y,t^\ast}^{-1/2}\sum_{\ell\in L}|Y(\ell)|.
\end{equation}
\end{lemma}
\begin{proof}
By a standard probabilistic argument (see \cite[Section 3]{Demeter-Wang}), it suffices to prove \eqref{furstenberg-Katz--Tao-esti} with the additional assumption that $L$ is a $(\de,t)$-set.

In this case, apply Theorem \ref{furstenberg-thm} to have
\begin{align}
    |E_L|&\gtrsim\de^\nu\de^{2}\de^{-\min\{t+s,\frac{3s+t}{2},1+s\}}\gtrsim\de^{\nu+O(\eta)}\de^{\max\{0,\frac{t-s}{2},t-1\}}\sum_{\ell\in L}|Y(\ell)|\\
    &\gtrsim \de^{\nu+O(\eta)}\de^{\max\{\frac{s-t}{2},\,0,\frac{t+s}{2}-1\}}\la^{1/2}\de^{(t-1)/2}\sum_{\ell\in L}|Y(\ell)|.
\end{align}
To prove \eqref{furstenberg-Katz--Tao-esti}, it suffices to show $\ga_{Y,t^\ast}\gtrsim \de^{O(\eta)} \de^{\min\{t-s,\,0,\,2-t-s\}}$.
But this is obviously true, as $|Y(\ell)|\gtrsim \de^{1-s+\eta}|N_\de(\ell)|$. \qedhere

\end{proof}

\smallskip

The next theorem was proved in \cite{Demeter-Wang}. We state a more quantitative version of it here.

\begin{theorem}[\cite{Demeter-Wang}, Theorem 5.4]
\label{DW-24}
For any $\nu>0$, there exists an $\eta>0$, which is much smaller than $\nu^{2\nu^{-1}}$, such that the following is true for sufficiently small $\de$: 

Let $\cT$ be a family of $\de\times 1$-tubes in the plane, and let $\cp$ be a family of uniform $\de$-balls such that for any $\rho\in[\de,1]$,
\begin{equation}
    |\cup_\cp|_\rho\gtrsim\rho^{-s}\de^{\eta}.
\end{equation}
For each $\de$-ball $p\in\cp$, let $\cT_p\subset\cT$ be a $(\de,s,\de^{-\eta})$-set (here we identify each $T\in\cT_p$ as a $\de$-arc contained in $\ZS^1$) of $\de\times1$-tubes passing through $p$ with $\#\cT_p\sim r$. 
Moreover, as a union of $\de$-arcs of $\ZS^1$, $\cup_{\cT_p}$ is uniform, and there is a uniform branching function of the family of the union of $\de$-arcs $\{\cup_{\cT_p}: p\in\cp\}$.

Then one of the following must be true:
\begin{enumerate}
    \item We have
    \begin{equation}
        \#\cp \lesssim\de^{s-\nu}\frac{(\#\cT)^2}{r^2}.
    \end{equation}
    \item There exists a scale $\De\gtrsim \de^{1-\sqrt{\eta}}$ such that for each $\De$-ball $B\subset (\cup_\cp)_\De$, the $\De^{-1}$-dilate of $\cup_\cp\cap B$ is a uniform, $(\de/\De, 2-s+\eta^{1/4}, \de^{-\eta})$-set.
\end{enumerate}

\end{theorem}

\begin{remark}
\rm

In  \cite[Theorem 5.4]{Demeter-Wang}, Item (2) is stated in a weaker way. 
However, their proof indeed gives the current stronger version of Item (2), as $\cp$ is uniform.

\end{remark}

We will use Theorem \ref{DW-24} to prove the following proposition, which is a combination of Lemma \ref{furstenberg-upper-range} and the dual version of Theorem \ref{DW-24}.

\begin{proposition}
\label{prop-induction}
For any $\nu>0$, there exists an $\eta>0$, which is much smaller than $\nu^{2\nu^{-1}}$, such that the following is true for sufficiently small $\de$:

Let $(L,Y)_\de$ be a set of uniform, $\de$-separated lines such that for all $\ell\in L$ we have
\begin{enumerate}
    \item $Y(\ell)$ is a $(\de,s, \de^{-\eta/2})$-set, yielding $|Y(\ell)|\gtrsim\de^{1-s+\eta/2}|N_\de(\ell)|$.
    \item $|Y(\ell)|\lesssim\de^{1-s-\eta/2}|N_\de(\ell)|$.
\end{enumerate}
Suppose in addition that
\begin{enumerate}[(i)]
    \item For each $\ell\in L$, the shading $Y(\ell)$ is uniform, and there is a uniform branching function for the family of shadings $\{Y(\ell): \ell\in L\}$.
    \item For all $\rho\in[\delta, 1]$, there are $\gtrsim \rho^{-s}\de^{\eta/2}$ many distinct $\rho \times 1$-tubes  $T_\rho$ such that $L[T_\rho]\not=\varnothing$. 
    \item $\#L(x)\lessapprox |E_L|^{-1}\sum_{\ell\in L}|Y(\ell)|$ for any $x\in E_L$.
\end{enumerate}
Then there exists a refinement $(L',Y')_\de$ of $(L,Y)_\de$ such that one of the following must be true for all $t\in[0,2]$:
\begin{enumerate}
    \item $\# L'(x)\lessapprox\de^{(s-t)/2}$ for all $x\in E_{L'}$. 
    \item We have
    \begin{equation}
        |E_{L'}|\gtrapprox \de^{\nu+3\eta}\de^{2-(t+3s)/2}.
    \end{equation}
    \item There is a $\De\gtrsim \de^{1-\sqrt{\eta}}$, and for each $\ell\in L'$, two numbers $d_1(\ell), d_2(\ell)$ and two families of disjoint $\De$-balls $\cq_1(\ell), \cq_2(\ell)$ such that the following is true:
    \begin{enumerate}
        \item $\#\cq_j(\ell)= d_j(\ell), j=1,2$, and $\cup_{\cq_2(\ell)}\subset\cup_{\cq_1(\ell)}\subset N_\De(\ell)$.
        \item $|E_L\cap Q|\gtrsim  \de^{1-s+2\eta}d_1(\ell)^{-1}\De^{-1}|Q|$ for each $Q\in\cq_1(\ell)$.
        \item For each $Q\in\cq_2(\ell)$, $Q$ contains $\gtrsim \de^{-s-\eta}d_2(\ell)^{-1}$ many $\de$-balls in $Y'(\ell)$.
    \end{enumerate} 
\end{enumerate}

\end{proposition}

\begin{remark}
\rm

The factor $d_2(\ell)$ appearing in Case (3) is not important to the numerological aspect of the proof of Theorem \ref{two-ends-furstenberg-general}. 
It was introduced to indicate that $Y'(\ell)$ is essentially contained in the set $\cup_{\cq_1(\ell)}$.
The reader may take $d_1(\ell)=d_2(\ell)$ and $\cq_1(\ell)=\cq_2(\ell)$ for convenience.
\end{remark}

\smallskip

\begin{proof}[Proof of Proposition \ref{prop-induction}]

Apply Lemma \ref{rich-point-refinement} to $(L,Y)_\de$ to obtain a refinement $(L,Y_1)$ of $(L,Y)_\de$ so that $\#L_{Y_1}(x)$ is the same up to a constant multiple for all $x\in E_{L,Y_1}$.
Apply Lemma \ref{uniformization} to $E_{L,Y_1}$ to obtain a uniform refinement $E_1$ of $E_{L,Y_1}$. 
Define $Y_2(\ell)=Y_1(\ell)\cap E_1$, so that $(L,Y_2)_\de$ is a refinement of $(L,Y)_\de$. 
Therefore, by dyadic pigeonholing, we can find a refinement $(L_2,Y_2)_\de$ of $(L,Y_2)_\de$ such that $L_2$ is uniform and 
\begin{enumerate}
    \item For all $\ell\in L_2$,  $Y_2(\ell)$ is uniform and is a refinement of $Y(\ell)$.
    As a result, $Y_2(\ell)$ is a $(\de,s, \de^{-\eta})$-set and $|Y_2(\ell)|\lessapprox\de^{1-s-\eta}|N_\de(\ell)|$ 
    \item For each $\ell\in L_2$, the shading $Y_2(\ell)$ is uniform, and there is a uniform branching function for the family of shadings $\{Y_2(\ell): \ell\in L_2\}$.
    \item For all $\rho\in[\delta, 1]$, there are $\gtrsim \rho^{-s}\de^{\eta}$ many distinct $\rho \times 1$-tubes  $T_\rho$ such that $L_2[T_\rho]\not=\varnothing$. 
\end{enumerate}

\smallskip

By the point-line duality, we can identify $L_2$ as a set of points $\cp\subset[0,1]^2$ with the following correspondences:
\begin{enumerate}
    \item Each $\ell\in L_2$ corresponds to a point $p\in \cp$, and the shading $Y_2(\ell)$ corresponds to a family of $\de\times1$-tubes $\cT_p$ passing through a point $p\in \cp$.
    \item The union of $\de$-balls $E_{L_2}$ corresponds to the union of $\de$-tubes $\cT=\cup_{p\in \cp} \cT_p$.
\end{enumerate}  
It is straightforward to check that the configuration $(\cp,\cup_{p\in\cp}\cT_p)$ obeys the hypothesis of Theorem \ref{DW-24}, with $r\gtrsim\de^{-s+\eta}$. 
Thus, one of the following must be true:
\begin{enumerate}
    \item[{\bf A.}] We have
    \begin{equation}
        \#\cp \lesssim\de^{s-\nu}\frac{(\#\cT)^2}{r^2}.
    \end{equation}
    \item[{\bf B.}] There exists a scale $\Delta\gtrsim \de^{1-\sqrt{\eta}}$ such that for each $\Delta$-ball $B\subset (\cup_\cp)_{\Delta}$, the $\Delta^{-1}$-dilate of $\cup_\cp\cap B$ is a uniform, $(\delta/\Delta, 2-s+\eta^{1/4}, \de^{-\eta})$-set.
\end{enumerate}

\smallskip

Suppose Case {\bf A} happens. 
Recall that $r\gtrsim\de^{-s+\eta}$.
By reversing point-line duality, we have
\begin{equation}
\label{DW-case-1}
    |E_{L_2}|\gtrsim \de^{\nu/2+\eta}\de^{2-3s/2}(\# L_2)^{1/2}.
\end{equation}
Recall that $\#L(x)\lessapprox |E_L|^{-1}\sum_{\ell\in L}|Y(\ell)|$ for all $x\in E_L$ and $|Y(\ell)|\lesssim\de^{1-s-\eta}|N_\de(\ell)|$ for any $\ell\in L$.
Since $(L_2,Y_2)_\de$ is a refinement of $(L,Y)_\de$, we have $\#L_2(x)\lessapprox |E_{L_2}|^{-1}\de^{2-s-\eta}(\# L_2)\lessapprox\de^{-\nu/2-2\eta}\de^{s/2}(\# L_2)^{1/2}$. 

Take $(L',Y')_\de=(L_2,Y_2)_\de$.
If $\# L\leq \de^{-t+3\eta+\nu}$, then $\# L_2(x)\lessapprox \de^{(s-t)/2}$ for any $x\in E_{L_2}$, giving item (1) in Proposition \ref{prop-induction}. 
If $\# L\geq \de^{-t+3\eta+\nu}$, then $|E_{L_2}|\gtrsim \de^{\nu+3\eta}\de^{2-(t+3s)/2}$, giving item (2) in Proposition \ref{prop-induction}.

\smallskip

Suppose Case {\bf B} happens. 
Then $L_2$ is the set of lines $L'$ we needed in the Proposition.
As for the rest, consider each $\Delta$-ball $B\subset(\cup_\cp)_{\Delta}$ and let $\cp_B$ be the set of $\delta$-balls in $\cp$ contained in $B$. 
For each $p\in \cp_B$, consider the family of tubes $\cT_p$, which is a $(\de,s, \de^{-\eta})$-set.
Since $\De\gtrsim\de^{1-\sqrt{\eta}}$, $\cT_p$ is a $(\de,s, (\de/\De)^{-\sqrt{\eta}})$-set for all $p\in\cp_B$.
Similarly, since the $\De^{-1}$-dilate of the $\cp_B$ is a $(\de/\De, 2-s+\eta^{1/4}, \de^{-\eta})$-set, it is also a $(\de/\De, 2-s+\eta^{1/4}, (\de/\De)^{-\sqrt{\eta}})$-set.
Applying Lemma \ref{furstenberg-upper-range} to the configuration $(\cp_B, \{\cT_p\}_{p\in \cp_B})$ with $(u,t,\eta,r)=(s, s-\eta^{1/4}, \eta^{1/2},\de^{-s+\eta})$, we have $\#\bigcup_{p\in \cp_B}\cT_p\gtrapprox (\De/\de)\de^{-s+\eta}$.

By reversing the point-line duality, each $\Delta$-ball $B$ corresponds to a $\Delta$-tube $T_{\Delta}$, and the $\de\times1$-tubes $\bigcup_{p\in \cp_B}\cT_p$ correspond to the $\de$-balls $\bigcup_{\ell\subset T_{\Delta}}Y_2(\ell)$.
Thus, what we had above can be stated as follows:
There is a partition of $L_2$ into $\{L_2[T_{\Delta}], T_{\Delta}\in\cT_{\Delta}\}$, where $\cT_{\Delta}$ is a family of $\Delta$-separated $\Delta\times1$-tubes, and $L_2[T_{\Delta}]$ is the set of $\ell\in L_2$ contained in $T_{\Delta}$, such that 
\begin{equation}
\label{big-at-Delta}
    \big|\bigcup_{\ell\in L_2[T_{\Delta}]}Y_2(\ell)\big|\sim\de^2\#\bigcup_{p\in \cp_B}\cT_p\gtrapprox(\Delta/\de)\de^{2-s+\eta}.
\end{equation}
Since $E_1$ is uniform, there exists a number $\al$ such that $|E_1\cap Q|\sim \al|Q|$ for all $Q\in\cd_{\De}(E_1)$.
For each $\ell\in L_2$, let $\cq_1(\ell)$ be the family of disjoint $\De$-balls contained in $N_\De(\ell)$ such that $E_1\cap Q\not=\varnothing$ for all $Q\in\cq_1(\ell)$.
Let $T_{\Delta}(\ell)$ be the $\Delta$-tube that contains $\ell$.
Since $\bigcup_{\ell'\in L_2[T_{\Delta}(\ell)]}Y_2(\ell')\subset E_1\cap N_\De(\ell)$, by \eqref{big-at-Delta}, we have
\begin{equation}
    \#\cq_1(\ell)\al\De^2\sim\sum_{Q\in\cq_1(\ell)}|E_1\cap Q|\sim |E_1\cap\cup_{\cq_1(\ell)}|\gtrapprox(\De/\de)\de^{2-s+\eta}.
\end{equation}
This gives item $(b)$ in Case $(3)$.

For items $(a), (c)$ in Case $(3)$, let $\cq_2(\ell)\subset\cq_1(\ell)$ be the family of $\De$-balls such that $Q\cap Y_2(\ell)\not=\varnothing$ for all $Q\in\cq_2(\ell)$. 
Since $Y_2(\ell)$ is uniform and since $Y_2(\ell)\subset E_1$, we know that $\cq_2(\ell)\subset \cq_1(\ell)$.
Finally, since $Y_2(\ell)$ is a uniform $(\de,s, \de^{-\eta})$-set, $Q$ contains $\gtrsim \de^{-s-\eta}d_2(\ell)^{-1}$ many $\de$-balls in $Y_2(\ell)$.
\qedhere

\end{proof}

\medskip 

Similar to the argument in \cite{Wang-Wu}, Theorem \ref{two-ends-furstenberg-general} reduces to the following proposition via a standard two-ends reduction.

\begin{proposition}
\label{two-ends-furstenberg-prop} 
Let $\e>0$ be given.
Let $\eta>0$ be such that Proposition \ref{prop-induction} and Lemma \ref{furstenberg-Katz--Tao-lem} are true with $(\nu,\eta)=(\e^2,2\eta)$.
Then there exists $c_\e>0$ such that the following is true for all $\delta \in (0, 1)$:

Let $(L, Y)_{\delta}$ be a set of $\delta$-separated lines in $\mathbb{R}^2$ with a uniform, $\lambda$-dense shading.
Let $t\in(0,2)$ and $t^\ast=\min\{t,2-t\}$.
Suppose $L$ is a Katz--Tao $(\de,t)$-set, and suppose $Y(\ell)$ is a $(\de,\e^2,C;\rho)$-set for some $\rho \in [\delta, \delta^{\eta}]$ and for all $\ell\in L$.
Then
\begin{equation}
\label{main-prop-esti}
    |E_L|\geq c_{\e} \delta^{\e} C^{-\eta^{-2}} \la^{1/2}\de^{(t-1)/2}\ga_{Y,t^\ast}^{-1/2}\sum_{\ell\in L}|Y(\ell)|. 
\end{equation}
\end{proposition}

\smallskip

\begin{proof}[Proof of Theorem \ref{two-ends-furstenberg-general} via Proposition \ref{two-ends-furstenberg-prop}]

It suffices to prove Theorem \ref{two-ends-furstenberg-general} for all $\e\in(0, \sqrt{\e_2})$.
By dyadic pigeonholing, there exists a refinement $L_1$ of $L$ such that for all $\ell\in L_1$, $|Y(\ell)|$ are the same up to a constant multiple. 
Let $Y_1=Y$, and assume $|Y_1(\ell)|\sim\la\delta$ for all $\ell\in L_1$, without loss of generality.

\smallskip

Apply Lemma \ref{uniformization} and Lemma \ref{uniform-sets-branching-lem} to the set of shadings $\{Y_1(\ell):\ell\in L_1\}$, we know that there is a refinement $(L_2,Y_2)_\de$ of $(L_1,Y_1)_\de$ such that 
\begin{enumerate}
    \item For all $\ell\in L_2$, $Y_2(\ell)$ is uniform, and $Y_2(\ell)$ is a refinement of $Y_1(\ell)$. 
    \item There is a uniform branching function $\be_{L_2}$ for the shadings $\{Y_2(\ell):\ell\in L_2\}$. 
\end{enumerate}
By Lemma \ref{two-ends-shading-lem}, there exists $K\lessapprox1$ such that $Y_2$ is $(\e_1, \e_2,K)$-two-ends.
For each $\ell\in L_2$, by Definition \ref{two-ends-reduction}, there exists a scale $\rho(\ell)=\rho(\ell;\e^2,K)\in[\delta,1]$ such that 
\begin{enumerate}
    \item $|Y_2(\ell)|_{\rho(\ell)}< K^{-1}\rho(\ell)^{-\e^2}$.
    \item For all $r\in[\de,\rho(\ell)]$ and all $J\subset (Y_2(\ell))_{\rho(\ell)}$, $|Y_2(\ell)\cap J|_r\gtrapprox K^{-1}(r/\rho)^{\e^2}$.
\end{enumerate}
Since $\e^2< \e_2$, by Lemma \ref{delta-1-rho}, $\rho(\ell)\geq\de^{\e_1}$ for all $\ell\in L_2$.
Apply dyadic pigeonholing to the set $\{\rho(\ell):\ell\in L_2\}$, we can find a uniform $\rho$ and a refinement $L_3$ of $L_2$ such that $\rho(\ell)=\rho$ for all $\ell\in L_3$.
Let $Y_3=Y_2$.
Then the set of lines and shading $(L_3,Y_3)_\de$ is a refinement of $(L,Y)_\de$ with a uniform, $\la$-dense shading.
Moreover, for all $\ell\in L_3$, the following is true:
\begin{enumerate}
    \item $|Y_3(\ell)|_\rho< K^{-1}\rho^{-\e^2} \lessapprox \rho^{-\e^2}$.
    \item For all $r\in[\de,\rho]$ and all $J\subset (Y_3(\ell))_{\rho}$, $|Y_3(\ell)\cap J|_r\gtrapprox K^{-1}(r/\rho)^{\e^2}$.
    \item $\rho\geq\de^{\e_1}$.
\end{enumerate}

\smallskip

Let $\cb_\rho$ be a family of finite-overlapping $\rho$-balls that covers $E_{L_3}$.
For each $B\in\cb_\rho$, let $\cT_B^0$ be the family of distinct $\de\times\rho$-tubes in $B$. 
Define $\cT_B$ as 
\begin{equation}
    \cT_B=\{J\in\cT_B^0:J\subset (Y_3(\ell))_\rho \text{ for some $\ell\in L_3$} \}.
\end{equation}
For each $J\in\cT_B$, let $L_3(J)=\{\ell\in L_3:J\subset(Y_3(\ell))_\rho \}$.
Pick one $\ell\in L_3(J)$ to define a shading $Y_B$ on $J$ as
\begin{equation}
    Y_B(J)=Y_3(\ell)\cap J.
\end{equation}
Since $Y_3$ is uniform, we know that  $Y_B(J)$ is uniform, and $|Y_B(J)|$ are the same up to a constant multiple for all $J\in\cT_B$.
Since $Y_3$ is $\gtrapprox\la$-dense and since $|Y_3(\ell)|_\rho  \lessapprox \rho^{-\e^2} $,  $Y_B(J)$ is $\gtrapprox\la  \rho^{-1 +\e^2}$-dense.
Recall that for all $r\in[\de,\rho]$ and all $T\subset (Y_3(\ell))_{\rho}$, $|Y_3(\ell)\cap T|_r\gtrapprox K^{-1}(r/\rho)^{\e^2}$.
Thus, the $\rho^{-1}$-dilate of $Y_B(J)$ is a $(\de/\rho,\e^2, K')$-set for some $K'\lessapprox1$.

\smallskip

We want to apply Proposition \ref{two-ends-furstenberg-prop} to the $\rho^{-1}$-dilate of $(\cT_B,Y_B)$. 
However,  the $\rho^{-1}$-dilate of $\cT_B$ may not be a Katz--Tao $(\de/\rho,t)$-set of tubes.
To overcome this issue, we will randomly pick a $\cT_B''\subset\cT_B$ that obeys that desired assumption.
After that, we will apply Proposition \ref{two-ends-furstenberg-prop} to $(\cT_B'',Y_B)$. 
Let us turn to details. 

\smallskip

Let $L_3(B)=\{\ell\in L_3: Y_3(\ell)\cap B\not=\varnothing\}$, so $L_3(B)=\bigcup_{J\in\cT_B}L_3(J)$.
Let $\mu_J=\#L_3(J)$.
By dyadic pigeonholing on $\{\mu_J:J\in\cT_B\}$, there is a subset $\cT_B'\subset\cT_B$ and a uniform number $\mu_B$ such that the following is true:
\begin{enumerate}
    \item $\mu_J\sim\mu_B$ for all $J\in\cT_B'$.
    \item $\mu_B\cdot \#\cT_B'\gtrapprox\sum_{J\in\cT_B}\mu_J\sim \#L_3(B)$.
\end{enumerate}

For each $v\in[\de,\rho]$, let $\cT_v(B)$ be the family of distinct $v\times\rho$-tubes in $B$, each of which contains at least one $J\in\cT_B'$.
Define a multiplicity factor $\si$, which measures how much a set fails to be a Katz--Tao set, as follows:
\begin{equation}
    \si:=\sup_{v\in[\de,\rho]} \{ (\de/v)^t\cdot\sup_{\tilde{J}\in\cT_v(B)}\#\{J\subset \tilde{J}: J\in\cT_B'\} \} \geq1.
\end{equation}
Note that for each $v\in[\de,\rho]$ and each $\tilde{J}\in\cT_v(B)$, the set of lines $L(\tilde{J})=\{\ell\in L:|N_\delta(\ell)\cap J| \gtrsim |J|\text{ for some }J\subset \tilde{J},\,J\in\cT_B'\}$ is contained in a $  (v/\rho)\times 1$-tube.
Since $L$ is a Katz--Tao $(\de,t)$-set, $\# L(\tilde{J})\lesssim (v/\de \rho)^t$.
Recall that there are $\gtrsim\#L_3(J)\sim\mu_B$ many $\ell\in L$ such that $|N_\delta(\ell)\cap J| \gtrsim |J|$ for each $J\in\cT_B'$.
Thus, for all $v\in[\de,\rho]$ and all $\tilde{J}\in\cT_v(B)$, $\mu_B^{-1}(v/\de \rho)^t\gtrsim\mu_B^{-1}\#L(\tilde{J})\gtrsim\#\{J\subset \tilde{J}: J\in\cT_B'\}$. 
This implies
\begin{equation}
\label{si-mu-B}
    \si \cdot \mu_B\lesssim \rho^{-t}.  
\end{equation}
Let $\cT_B''\subset\cT_B'$ be a uniform random sample of probability $\si^{-1}$.
Thus, with high probability, $\cT_B''\gtrapprox \si^{-1}\cT_B'$, and $\tilde{J}$ contains $\lessapprox(v/\de)^t$ many $\de\times \rho$-tubes in $\cT_B''$ for all $\tilde{J}\in\cT_v(B)$ and all $v\in[\de,\rho]$.
This shows that the $\rho^{-1}$-dilate of $\cT_B''$ is a Katz--Tao $(\de/\rho,t)$-set.

\smallskip

Recall that $Y_B$ is a uniform, $\gtrapprox\la\rho^{-1+\e^2}$-dense shading, and the $\rho^{-1}$-dilate of $Y_B(J)$ is a $(\de/\rho,\e^2, K')$-set for all $J\in\cT_B$ and some $K'\lessapprox1$.
Apply the $\rho$-dilate version of Proposition \ref{two-ends-furstenberg-prop} to $(\cT_B'',Y_B)$  (with $\e/2$ in place of $\e$) so that
\begin{equation}
\label{scale-rho-esti}
    |E_{\cT_B'',Y_B}|\gtrapprox c_{\e}(\delta/\rho)^{\e/2}\cdot (\lambda \rho^{-1+\e^2})^{1/2}(\de/\rho)^{(t-1)/2}\ga_{Y_B,t^\ast}^{-1/2}\sum_{J\in\cT_B''}|Y_B(J)|. 
\end{equation}

Recall that $\cT_B''\gtrapprox\si^{-1}\cT_B'$, $\#L_3(J)\sim\mu_B$ for all $J\in\cT_B'$, and $\mu_B\cdot \#\cT_B'\gtrapprox \#L_3(B)$.
Since $|Y_B(J)|$ are the same up to a constant multiple, by \eqref{si-mu-B},
\begin{align}
    \sum_{J\in\cT_B''}|Y_B(J)|&\gtrapprox\si^{-1}\sum_{J\in\cT_B'}|Y_B(J)|\gtrapprox(\si\mu_B)^{-1}\sum_{\ell\in L_3(B)}|Y_3(\ell)\cap B|\\
    &\gtrapprox\rho^{t}\int_B\#L_3(x).
\end{align}
Since $\rho\geq\de^{\e_1}$ and since $E_{\cT_B'',Y_B}\subset E_L\cap B$, plug this back to \eqref{scale-rho-esti} so that
\begin{align}
\label{ELcapB-two-ends-reduction}
    |E_L\cap B|&\gtrapprox  c_{\e}(\delta/\rho)^{\e/2} \cdot (\lambda \rho^{-1+\e^2})^{1/2}(\de/\rho)^{(t-1)/2}\ga_{Y_B,t^\ast}^{-1/2}\cdot\rho^{t}\int_B\#L_3(x)\\
    &\geq c_{\e}(\delta/\rho)^{\e/2}  \cdot    \rho^{\e^2/2}\de^{t\e_1/2}\la^{1/2}\de^{(t-1)/2}\ga_{Y_B,t^\ast}^{-1/2}\int_B\#L_3(x).
\end{align}

Notice that $\ga_{Y_B,t^\ast}\lesssim \ga_{Y,t^\ast}$ for all $B\in\cb_\rho$.
Since $\cb_\rho$ covers $E_{L_3}$ and since $(L_3,Y_3)_\de$ is a refinement of $(L,Y)_\de$, we sum up all $B\in\cb_\rho$ in \eqref{ELcapB-two-ends-reduction} to get
\begin{align}
    |E_L|&\gtrapprox c_{\e}\de^{\e/2}   \cdot    \de^{t\e_1/2}\la^{1/2}\de^{(t-1)/2}\ga_{Y,t^\ast}^{-1/2}\int\#L_3(x)\\
    &\geq c_\e\de^{\e}\de^{t\e_1/2}\la^{1/2}\de^{(t-1)/2}\ga_{Y,t^\ast}^{-1/2}\sum_{\ell\in L}|Y(\ell)|. \qedhere
\end{align}

\end{proof}

\medskip 

\subsection{Proof of the main proposition}

Finally, let us see how Proposition \ref{prop-induction} and Lemma \ref{furstenberg-Katz--Tao-lem} imply Proposition \ref{two-ends-furstenberg-prop}.

\begin{proof}[Proof of Proposition \ref{two-ends-furstenberg-prop}]

We will prove \eqref{main-prop-esti} by a backward induction on $\de$.
In the base case when  $\delta\gtrsim_{\e} 1$, we choose $c_{\e}$ sufficiently small such that \eqref{main-prop-esti} is true.
By dyadic pigeonholing, there exists a refinement $L_1$ of $L$ such that for all $\ell\in L_1$, $|Y(\ell)|$ are the same up to a constant multiple. 
Without loss of generality, assume
\begin{equation}
    |Y(\ell)|\sim\la\delta
\end{equation}
for all $\ell\in L_1$.
Denote by $Y_1=Y$, so that $(L_1,Y_1)_\de$ is a refinement of $(L,Y)_\de$ and a family of $\de$-separated lines with a uniform,  $\lambda$-dense shading.
Moreover, there exists a $\rho\in[\de,\de^{\eta}]$ such that $Y_1(\ell)$ is a $(\de,\e^2,C;\rho)$-set for all $\ell\in L_1$.
Note that when $C\geq \delta^{-2\eta^2}$, $|E_L|\geq\sup_{\ell\in L}|Y(\ell)|\geq\de^2\geq C^{-\eta^{-2}}$, which implies \eqref{main-prop-esti} directly.
Thus, we can assume $C\leq\de^{-2\eta^2}$, in which case the spacing condition on $Y_1$ is not vacuous, as $\rho\leq\de^{\eta}$.

\smallskip

Applying Lemma \ref{uniformization} and Lemma \ref{uniform-sets-branching-lem} to the set of shadings $\{Y_1(\ell)\}_{\ell\in L_1}$, we know that there is a refinement $(L_2,Y_2)_\de$ of  $(L_1,Y_1)_\de$ such that $\{Y_2(\ell):\ell\in L_2\}$ admits a uniform branching function.
Let 
\begin{equation}
    \eta_1:=\eps\e^3.
\end{equation}
Let $t^\ast=\min\{t,2-t\}$.
Apply  Lemma \ref{multi-scale-lem} to this family of shadings with $\eta_0=\eta_0(\eta_1)$ with $t=t^\ast$ to obtain a scale $r$ such that the following is true: 
\begin{enumerate}
    \item $r>\delta^{1-\eta_0\eta_1^{-1}}$.
    \item $\ga_{\tilde Y_2, t^\ast}(\ell)\lesssim \ga_{Y_2, t^\ast}(\ell)$, where $\tilde{Y}_2(\ell) = (Y_2(\ell))_r$. 
    \item For each $\delta \times r$-tube $J\subset (Y_2(\ell))_r$, the $r^{-1}$-dilate of $Y_2(\ell)\cap J$ along $\ell$ is a $(\delta/r, s, (\delta/r)^{-9\eta_1})$-set, and
    $\log_{1/\de}\big(\frac{|Y_2(\ell)|_{\de}}{|Y_2(\ell)|_{r}}\big)\leq (s+9\eta_1) \log_{1/\de}(r/\de)$.
\end{enumerate}
Since $Y_2(\ell)$ is a refinement of $Y(\ell)$, by Lemma \ref{refinement-delta-s-Delta-set}, $Y_2(\ell)$ is a $(\de,\e^2,CC_2;\rho)$-set for some $C_2\lessapprox1$.

\medskip  

We consider two separate cases: $r\geq \de^\eta$ and $r\leq \de^{\eta}$.
When $r\geq \de^{\eta}$, note that for all $\ell\in L_2$, $Y_2(\ell)$ is a $(\de,s,\de^{-2\eta})$-set, $|Y_2(\ell)|\approx\la|N_\delta(\ell)|$, and $|Y_2(\ell)|\lesssim \de^{1-s-2\eta}|N_\de(\ell)|$.
Recall that $(\e,\eta)$ is chosen such that Lemma \ref{furstenberg-Katz--Tao-lem} is true for $(\nu,\eta)=(\e^2,2\eta)$.
Apply Lemma \ref{furstenberg-Katz--Tao-lem} to $(L_2,Y_2)_\de$ to have
\begin{equation}
    |E_{L_2}|\geq \de^{\e^2+O(\eta)}\la^{1/2}\de^{(t-1)/2}\ga_{Y_2,t^\ast}^{-1/2}\sum_{\ell\in L_2}|Y_2(\ell)|.
\end{equation}
This concludes \eqref{main-prop-esti} as $(L_2,Y_2)_\de$ is a refinement of $(L,Y)_\de$.

\smallskip

From now on, we assume $r\leq\de^\eta$ (we are only going to use this assumption in Step 5, Case {\bf C}).

\medskip

\noindent {\bf Step 1: Uniformization of $E_L$ inside each $r$-ball}

We will establish a uniformization on $E_{L}\cap B$ for each $r$-ball $B\subset(E_L)_r$. 
Define
\begin{equation}
\label{tilde-lambda}
    \tilde\la:=\la (r/\de)^{1-s-\eta_1}.
\end{equation}
Let $\cT_{\de/r}$ be a maximal collection of distinct $(\de/r)\times1$-tubes. 
For each $T\in\cT_{\de/r}$, recall Definition \ref{L[T]-def} for $L_2[T]$,
and let $\tilde L_2[T]=L_2[T]$.
For each $\ell\in L_2$, define a new shading  $\tilde Y_2(\ell) = (Y_2(\ell))_r$ by $\delta\times r$-tubes as in Definition \ref{tube-segment}.
Note that for each $\de\times r$-tube $J\subset \tilde Y_2(\ell)$, we have $|Y_2(\ell)\cap J|\lesssim (\de/r)^{1-s-\eta_1}|N_\de(\ell) \cap J|$.  
Consequently, $\tilde Y_2(\ell)$ is $\gtrapprox \tilde\la$-dense, as $Y_2(\ell)$ is $\gtrapprox \lambda$-dense. 
Apply Lemma \ref{rich-point-refinement} to $(\tilde L_2[T], \tilde Y_2)_\de$ to obtain a refinement $(\tilde L_2'[T], \tilde Y_2')_\de$ of $(\tilde L_2[T], \tilde Y_2)_\de$ satisfying 
\begin{equation}
    \#(\tilde L_2[T])_{\tilde Y_2'}(x)\sim\frac{\sum_{\ell\in \tilde L_2[T]}|\tilde Y_2(\ell)|}{|E_{\tilde L_2[T], \tilde Y_2'}|}\approx\frac{\sum_{\ell\in \tilde L_2'[T]}|\tilde Y_2'(\ell)|}{|E_{\tilde L_2'[T]}|}.
\end{equation}
Note that $\tilde Y_2'(\ell)$ is $\gtrapprox\tilde\la$-dense and is still a union of $\delta\times r$-tubes.

Recall that $Y_2(\ell)$ is a $(\de,\e^2,CC_2;\rho)$-set for some $C_2\lessapprox1$.
Since $Y_2$ is uniform, $\tilde Y_2(\ell)$ is a $(r,\e^2,C\tilde C_2;\max\{\rho,r\})$-set for some $\tilde C_2\lessapprox1$.
Thus, the anisotropic $r/\de$-dilate of $(\tilde L_2'[T], \tilde Y_2')_\de$ is a $(r,\e^2,C\tilde C_2;\max\{\rho,r\})$-set.
Clearly, $\max\{\rho,r\}\in[r,r^\eta]$, as $\rho\in[\de,\de^\eta]$.
Apply Proposition \ref{two-ends-furstenberg-prop} at scale $r$ to this anisotropic $r/\de$-dilate to bound $E_{\tilde L_2'[T]}$ so that
\begin{align}
    \mu_T:=\frac{\sum_{\ell\in \tilde L_2'[T]}|\tilde Y_2'(\ell)|}{|E_{\tilde L_2'[T]}|}&\lessapprox c_\e^{-1} r^{-\e}(C\tilde C_2)^{\eta^{-2}} \tilde\la^{-1/2} r^{(1-t)/2}\ga_{\tilde Y_2', t^\ast}^{1/2}\\
    &\leq c_\e^{-1} r^{-\e}(C\tilde C_2)^{\eta^{-2}} \tilde\la^{-1/2} r^{(1-t)/2}\ga_{ Y_2, t^\ast}^{1/2}.
\end{align}
We use the estimate $\ga_{\tilde Y_2', t^\ast}\leq\ga_{\tilde Y_2, t^\ast}\lesssim \ga_{Y_2, t^\ast}$ in the last inequality.   
Since the intersection of $N_\de(\ell_1)$ and $N_\de(\ell_2)$ contains at least one $\de\times r$-tube for any two $\ell_1,\ell_2\in L_2[T]$ with $\tilde Y_2'(\ell_1)\cap \tilde Y_2'(\ell_2)\neq \emptyset$, the above estimate shows that for each $\de\times r$-tube $J\subset E_{\tilde L_2( T), \tilde Y_2'}$ parallel to $T$, 
\begin{align} 
\label{it: muT}
    \#\{\ell\in \tilde L_2[T], J\subset \tilde Y_2'(\ell)\}\approx\mu_T
    \lessapprox c_\e^{-1} r^{-\e}(C\tilde C_2)^{\eta^{-2}} \tilde\la^{-1/2} r^{(1-t)/2}\ga_{ Y_2, t^\ast}^{1/2}
\end{align}

\smallskip 

Let $\tilde L_2=\bigcup_{T\in\cT_{\delta/r}}\tilde L_2[T]$ and let $Y_3'(\ell)=\tilde Y_2'(\ell)\cap Y_2(\ell)$ for any $\ell\in \tilde L_2'$. Since $|Y_2(\ell)\cap J|$ are about the same up to a constant multiple, $(\tilde L_2,Y_3')_\de$ is a refinement of  $(L_2,Y_2)_\de$.
By  dyadic pigeonholing, there exists a uniform
\begin{equation}
\label{small-multi-assumption}
    \mu_3\lessapprox  c_\e^{-1} r^{-\e}(C\tilde C_2)^{\eta^{-2}} \tilde\la^{-1/2} r^{(1-t)/2}\ga_{ Y_2, t^\ast}^{1/2}
\end{equation}
and a refinement  $(L_3,Y_3')_\de$ of $(\tilde L_2,Y_3')_\de$ such that the following holds. 
For any $T\in \cT_{\delta/r}$ with $L_3[T]$ nonempty, we have $\mu_T\sim \mu_3$.   
It follows from  \eqref{it: muT} that for any $\delta \times r$-tube $J\subset T$ parallel to $T$ and $ E_{L_3[T], Y_3'}\cap J\neq \varnothing$,  
\begin{equation}
\label{uniform-multi-assumption-pre}
    \#\{\ell\in L_3[T]:  Y_3'(\ell)\cap J\not=\varnothing\}\approx\mu_3.
\end{equation}

\smallskip

Note that for each $r$-ball $B$, $\tilde Y_2'(\ell)\cap B$ is either an empty set, or is essentially an $\de\times r$-tube.
Let $\cb_{r,3}$ be the set of finite-overlapping $r$-balls contained in $(E_{L_3, Y_3'})_r$.  
For each $B\in\cb_{r,3}$, consider the family of distinct  $\de\times r$-tubes 
\begin{equation}
    \cT_B^0:=\{J: \tilde Y_2'(\ell)\cap B\subset J\text{ for some $\ell\in L_3$}\}.
\end{equation}
Thus, if $Y_3'(\ell)\cap B\not=\varnothing$, then there exists a $\de\times r$-tube $J\in\cT_B^0$ such that $Y_3'(\ell)\cap B\subset J$.
By Lemma \ref{uniformization}, there is a uniform refinement $\cT_B$ of $\cT_B^0$.

Now for each $\ell\in L_3$, define a new shading $Y_3(\ell)\subset Y_3'(\ell)$ as follows: 
First, for each $r$-ball $B\in\cb_{r,3}$ such that $Y_3'(\ell)\cap B\not=\varnothing$, define
\begin{equation}
    Y_3(\ell)\cap B=\left\{\begin{array}{cc}
    Y_3'(\ell)\cap B,  & \text{ if $\exists J\in\cT_B$ such that $\tilde Y_2'(\ell)\cap B\subset J$;} \\
    \varnothing,   & \text{otherwise}.
    \end{array}\right.
\end{equation}
Then, define $Y_3(\ell)=\bigcup_{B\in\cb_{r,3}}Y_3(\ell)\cap B$.

Since $\cT_B$ is a refinement of $\cT_B^0$, it follows from  \eqref{uniform-multi-assumption-pre} that for each $B\in\cb_{r,3}$,
\begin{equation}
    \int_B \#(L_3)_{Y_3'}(x)\lessapprox\int_B \#(L_3)_{Y_3}(x).
\end{equation}
This shows that $(L_3, Y_3)_\de$ is a refinement of $(L_3, Y_3')_\de$.

Since $Y_3(\ell)$ is a refinement of $Y_2(\ell)$ at resolution $r$: for  all $\ell\in L_3$ and all $\de\times r$-tubes $J\subset (Y_3(\ell))_r$,  $Y_3(\ell)\cap J = Y_2(\ell)\cap J$, 
the $r^{-1}$-dilate of $Y_3(\ell)\cap J$ along $\ell$ is a $(\de/r, s, (\de/r)^{-\eta_1})$-set, and $|Y_3(\ell)\cap J|$ are about the same up to a constant multiple.

As a remark, we remind the reader that the following is true:
For all $\ell\in L_3$ and all $\de\times r$-tubes $J\subset (Y_3(\ell))_r$, we have
\begin{enumerate}
    \item $Y_3(\ell)\cap J = Y_2(\ell)\cap J$.
    \item The $r^{-1}$-dilate of $Y_3(\ell)\cap J$ along $\ell$ is an $(\de/r, s, (\de/r)^{-\eta_1})$-set.
    \item $|Y_3(\ell)\cap J|\lesssim (\de/r)^{1-s-\eta_1}|J|$.
    \item $|Y_3(\ell)\cap J|$ are about the same.
\end{enumerate}
Moreover, for each $B\in\cb_{r,3}$, there is a uniform set of $\de\times r$-tubes $\cT_B$ such that for all $\ell\in L_3$, $Y_3(\ell)\cap B\subset J$ for some $J\in\cT_B$.
Also, by \eqref{uniform-multi-assumption-pre}, for all $J\in\cT_B$, we have
\begin{equation}
\label{uniform-multi-assumption}
    \#\{\ell\in L_3:  J\subset (Y_3(\ell))_r\}\approx\mu_3.
\end{equation}

\medskip

\noindent {\bf Step 2: A broad-narrow argument.}

Applying Lemma \ref{broad-narrow-lem} to $(L_3, Y_3)_\de$ (with $\al$ in place of $\rho$), after dyadic pigeonholing, there is an $\al\in[\de,1]$ and a set $E_\al\subset E_{L_3}$ such that
\begin{enumerate}
    \item For each $x\in E_\al$, there exists a $10\al$-cap $\si_x\subset\ZS^1$ and a refinement $L_3'(x)$ of $L_3(x)$ such that the direction $V(\ell)\in \si_x$ for all $\ell\in L_3'(x)$.
    \item There are two subsets of lines $L', L''\subset L_3'(x)$ such that $\# L', \# L''\gtrapprox \# L_3'(x)$, and $\al\geq\ang(\ell',\ell'')\gtrapprox \al$ for all $\ell'\in L',\ell''\in L''$.
\end{enumerate}
Define for each $\ell\in L_3$ a new shading $Y_4(\ell)=E_\al\cap Y_3(\ell)$, so $(L_3, Y_4)_\de$ is a refinement of $(L_3, Y_3)_\de$. 
Since $(L_3, Y_3)_\de$ is a refinement of $(L_2, Y_2)_\de$, $(L_3, Y_4)_\de$ is a refinement of $(L_2, Y_2)_\de$.
Thus, by dyadic pigeonholing, there exists a refinement $(L_4, Y_4)_\de$ of $(L_3, Y_4)_\de$ (so it is a refinement of $(L_2,Y_2)_\delta$ as well) such that $Y_4(\ell)$ is a refinement of $Y_2(\ell)$ for all $\ell\in L_4$.

\smallskip

Suppose $\al\leq (\de/r)^{\eta_1}$. 
Recall that $Y_2(\ell)$ is uniform for all $\ell\in L_2\supset L_4$.
Since for each $\ell\in L_4$, $Y_4(\ell)$ is a refinement of $Y_2(\ell)$, by Lemma \ref{uniformization}, there exists a new shading $Y_4'$ such that $Y_4'(\ell)$ is uniform and is a refinement of $Y_4(\ell)$ for all $\ell\in L_4$. 
For each $\ell\in L_4$, consider a new shading $\tilde Y_4'(\ell)=(Y_4'(\ell))_{\de/\al}$.
Since $Y_4'(\ell)$ is uniform, there exists a $\kappa(\ell)<1$ such that the following is true:
\begin{enumerate}
    \item $|Y_4'(\ell)|\approx \kappa(\ell)|\tilde Y_4'(\ell)|$.
    \item $|Y_4'(\ell)\cap J|\approx \kappa(\ell)|J|$ for each $\de\times (\de/\al)$-tube $J\subset (Y_4'(\ell))_{\de/\al}$.
\end{enumerate}
By dyadic pigeonholing on $\{\kappa(\ell): \ell\in L_4\}$, there exists a uniform $\kappa<1$ and a set of lines $L_4'\subset L_4$ such that the following is true. 
\begin{enumerate}
    \item $(L_4',Y_4')_\de$ is a refinement of $(L_4,Y_4)_\de$.
    \item $\kappa(\ell)\sim \kappa \leq 1$ for all $\ell\in L_4'$.
\end{enumerate}

\smallskip

Let $\cT_\alpha$ be a collection of $\alpha$-separated $\alpha\times 1$-tubes, and for each $T\in\cT_{\alpha}$, recall Definition \ref{L[T]-def} for $L_4'[T]$.
Then
\begin{equation}
\label{narrow}
    |E_{L_4}|\gtrsim\sum_{T\in\cT_{\alpha}}|E_{L_4'[T]}|
\end{equation}
For each $T\in \cT_{\alpha}$, let $\phi_T$ be the anisotropic $\alpha^{-1}$-dilate which maps  $T$ to the unit ball, so  $|\phi_T(E_{L_4'[T]})| = (\alpha/\delta) |E_{L_4[T]}|$.
For all $\ell\in L_4'\subset L_2$, recall that $Y_2(\ell)$ is a $(\de,\e^2,CC_2;\rho)$-set for some $C_2\lessapprox1$.
Since $Y'_4(\ell)$ is a refinement of $Y_2(\ell)$, $Y_4'(\ell)$ is a $(\de,\e^2,CC_4;\rho)$-set for some $C_4\lessapprox1$.
Thus, $\tilde Y_4'(\ell)$ is a $(\de/\al,\e^2,CC_4;\max\{\rho,\de/\al\})$-set, 
implying that $\phi_T(\tilde Y_{4}')$ is a $(\de/\al,\e^2,CC_4;\max\{\rho,\de/\al\})$-set.
Clearly, $\max\{\rho,\de/\al\}\in[\de/\al,(\de/\al)^\eta]$, as $\rho\in[\de,\de^\eta]$.
Moreover, since $Y_4'(\ell)$ is $\gtrapprox\la$-dense, we have that $\tilde Y_{4}'(\ell)$ is $\gtrapprox\la \kappa^{-1} $-dense, which implies $\phi_T(\tilde Y_{4}')$ is $\gtrapprox\la \kappa^{-1} $-dense.
Apply induction at scale $\de/\alpha$ to $(\phi_{T}(L_{4}'[T]),\phi_T(\tilde Y_{4}'))_{\de/\alpha}$ so that, after rescaling back, we have 
\begin{align}
    |E_{L_4'[T], \tilde Y_{4}'}|&\geq c_\e(\de/\al)^\e(CC_4)^{-\eta^{-2}}\la^{1/2}\ka^{-1/2}(\de/\al)^{(t-1)/2}\ga_{\tilde Y_4', t^\ast}^{-1/2}\sum_{\ell\in L_4'[T]}|\tilde Y_{4}'(\ell)|\\
    &\gtrapprox c_\e(\de/\al)^\e(CC_4)^{-\eta^{-2}}\la^{1/2}\ka^{-1}\de^{(t-1)/2}\ga_{Y_4', t^\ast}^{-1/2}\sum_{\ell\in L_4'[T]}| Y_{4}'(\ell)|.
\end{align}
In the last inequality, we use $\ga_{\tilde Y_4',t^\ast}\leq \al^{1-t^\ast}\ka^{-1}\ga_{Y_4',t^\ast}$, a consequence of Lemma \ref{gamma-y-lemma} (note that when $t>1$, $\al^{1-t^\ast}=\al^{t-1}$).

Since $|E_{L_4[T]}| \gtrsim \ka |E_{L_4'[T], \tilde Y_{4}'}|$, we thus have
\begin{equation}
    |E_{L_4[T]}|\gtrapprox c_\e(\de/\al)^\e(CC_4)^{-\eta^{-2}}\la^{1/2}\de^{(t-1)/2}\ga_{Y_4', t^\ast}^{-1/2}\sum_{\ell\in L_4'[T]}| Y_{4}'(\ell)|.
\end{equation}
Finally, since $(L_4', Y_4')_\de$ is a refinement of $(L,Y)_\de$ and since $\ga_{Y_4',t^\ast}\leq\ga_{Y,t^\ast}$, by \eqref{narrow}, the above estimate implies
\begin{equation}
    |E_{L}|\geq|E_{L_4}|\geq c_\e\de^{\e}C^{-\eta^{-2}} \la^{1/2}\de^{(t-1)/2}\ga_{Y,t^\ast}^{-1/2}\sum_{\ell\in L}|Y(\ell)|.
\end{equation}

\smallskip

Suppose $\al\geq(\de/r)^{\eta_1}$. 
We proceed to the next step.

\medskip

\noindent {\bf Step 3: Setting up an incidence problem inside an $r$-ball.}

For $k=3,4$ and each $\ell\in L_k$,  let $\cj_k(\ell)=\{J\subset(Y_k(\ell))_r\}$.
Recall that $Y_4(\ell)$ is a $\gtrapprox 1 $-refinement of $Y_2(\ell)$. 
Since $Y_4(\ell)\subset Y_3(\ell)\subset Y_2(\ell)$ and since $Y_2(\ell)$ is uniform, we have $\#\cj_4(\ell)\leq\#\cj_3 (\ell)\lessapprox\#\cj_4(\ell)$.

Let $\cj_5(\ell)=\{J\in\cj_4(\ell): |Y_4(\ell)\cap J|\gtrapprox|Y_3(\ell)\cap J|\}$. 
Recall that  $|Y_3(\ell)\cap J|=|Y_2(\ell)\cap J|$ are  the same up to a constant multiple for all $J\in\cj_3(\ell)$. 
Thus, since $Y_4(\ell)$ is a $\gtrapprox 1 $-refinement of $Y_3(\ell)$, we have $\#\cj_5(\ell)\gtrapprox\#\cj_4(\ell)$. 
Consider a new shading  $Y_5(\ell)=\bigcup_{J\in\cj_5(\ell)}J\cap Y_4(\ell)$, so $Y_5(\ell)$ is a $\gtrapprox 1 $-refinement of $Y_4(\ell)$. 
Consequently, by taking $L_5=L_4$, $(L_5,Y_5)_\de$ is a refinement of $(L_4,Y_4)_\de$.

We remark that the set of lines and shading $(L_5,Y_5)_\de$ has the following properties
\begin{enumerate}
    \item For any $x\in E_{L_5}$, $(L_{3})_{Y_3}(x)$  is two-broad (i.e. $\al \geq (\delta/r)^{\eta_1}$, assumed in Step 2), since $E_{L_5}\subset E_{L_4}\subset E_\al$.
    \item For any $\ell\in L_5$ and any $\de\times r$-tube $J\in\cj_5(\ell)$, the (1-dimensional) $r^{-1}$-dilate of $ Y_5(\ell)\cap J$ along the line $\ell$ is a $(\de/r, s, (\de/r)^{-2\eta_1})$-set. This is a consequence of Lemma \ref{refinement-delta-s-set} and the fact that the $r^{-1}$-dilate of  $Y_3(\ell)\cap J$ is a $(\de/r, s, (\de/r)^{-\eta_1})$-set (see the end of Step 1).
\end{enumerate}

\smallskip

Let $\cb_{r,5}\subset\cb_{r,3}$ be the family of $r$-balls contained in $(E_{L_5, Y_5})_r$.
For each $r$-ball $B\in\cb_{r,5} \subset \cb_{r, 3}$, let $\cT_B$ be the set of distinct $\delta\times r$-tube segments defined in the end of Step 1.  
We restate some properties of $\cT_B$ here:
\begin{enumerate}
    \item For each $\ell \in L_3$ with $Y_3(\ell)\cap B\neq \varnothing$, there exists $J\in \cT_B$ such that $J\subset (Y_3(\ell))_r$.
    \item For any two $J, J'\in \cT_B$, $|J\cap J'|\leq |J|/2$.
    \item $\cT_B$ is uniform, and for each $J\in\cT_B$, \eqref{uniform-multi-assumption} is true.
\end{enumerate}
We are going to analyze the incidences between $\delta\times r$-tubes in $\cT_B$ and $\delta$-balls in $B\cap E_{L_3, Y_3}$ (here we choose $E_{L_3, Y_3}$ over $E_{L_5, Y_5}$ because of the following: Each $x\in E_{L_5, Y_5}$ is ``two-broad" with respect to $(L_3,Y_3)_\de$ but not with respect to $(L_5,Y_5)_\de$).

\smallskip

First, we study the structure of $\cT_B$. 
Pick $\ell\in L_5$ such that $Y_5(\ell)\cap B\not=\varnothing$, so the $r^{-1}$-dilate of $ Y_5(\ell)\cap B$ along the line $\ell$ is a $(\de/r, s, (\de/r)^{-2\eta_1})$-set. 
Recall the ``two-broad" property obtained in Step 2: 
For each $x\in E_{L_5} \subset E_{\al}$, there exists $\ell'\in L_3$ such that $ x\in Y_3(\ell')$ and $\angle (\ell, \ell')\gtrapprox \al\geq(\de/r)^{\eta_1} $.
Hence, for each $v\in[\de,r]$, the number of distinct $v\times r$-tubes required to cover $\cT_B$ is $\gtrsim (r\al/v)^{s-2\eta_1}\geq (r/v)^{s}(\de/r)^{3\eta_1}$. 
Let $\cT_v(B)$ denote this set of $v\times r$-tubes.

Next, for each $J\in\cT_B$, define
\begin{equation}
    L_3(J):=\{\ell\in L_3: J\subset(Y_3(\ell))_r\},
\end{equation}
and let $L_3(B)=\bigsqcup_{J\in\cT_B}L_3(J)$ be the set of lines whose shading intersects $B$. 
Define $K_J=\#L_3(J)$, so by \eqref{uniform-multi-assumption}, $K_J\approx \mu_3$. 
Let $\ell_1(J),\ldots, \ell_{K_J}(J)$ be an enumeration of the lines in $L_3(J)$.
Such enumeration on each $L_3(J)$ gives a natural disjoint partition of $L_3(B)$: 
Let $L_{3,k}(B)=\{\ell_{k}(J):J\in\cT_B\}$, where $\ell_{k}(J):=\varnothing$ if $k> K_J$.
For each $k$ and each $J\in\cT_B$, define a shading
\begin{equation}
\label{shading-B-k}
    Y_{B,k}(J)=\left\{\begin{array}{cc}
    Y_3(\ell_k(J)) \cap J,  & \text{ if } k \leq K_J \\
    \varnothing,   & \text{otherwise}.
    \end{array}\right.
\end{equation}
Let $K_B = \min_{ J\in \cT_B} K_J\approx \mu_3$, so when $k\leq K_B$, $Y_{B,k}(J)\not=\varnothing$ for all $J\in\cT_B$.
Since $|Y_3(\ell)\cap J|$ are about the same for all $\ell\in L_3(J)$ and all $J\subset(Y_3(\ell))_r$, we know that $\sum_{J\in\cT_B}|Y_{B,k}(J)|\approx\sum_{J\in\cT_B}|Y_{B,k'}(J)|$ when $k,k'\leq K_B$, and $\sum_{J\in\cT_B}|Y_{B,k}(J)|\gtrapprox\sum_{J\in\cT_B}|Y_{B,k'}(J)|$ when $k\leq K_B\leq k'$.
Notice that $Y_{B,k}=\varnothing$ whenever $k>\max_{J\in\cT_B}K_J\approx\mu_3\approx K_B$.
Consequently,
\begin{equation}
\label{a-lot-of-J}
    \int_{B}\#L_3(x)=\sum_k\sum_{p\subset E_{\cT_B, Y_{B,k}}}\#\cT_B(p)\lessapprox\sum_{k\leq K_B}\sum_{p\subset E_{\cT_{B}, Y_{B,k}}}\#\cT_{B}(p).
\end{equation}
\medskip

\noindent {\bf Step 4.1: Analyze $E_L\cap B$ for each $r$-ball $B$ in the broad case (I).}

Recall that for all $B\in\cb_{r,5}$ and all $k\leq K_B$,  the configuration $(\cT_B, Y_{B,k})$ satisfies the following properties:
\begin{enumerate}
    \item $\cT_B$ is uniform.
    \item For each $v\in[\de,r]$, there are $\gtrsim (r/v)^{s}(\de/r)^{3\eta_1}$ distinct $v\times r$-tubes containing at least one $J\in\cT_B$. 
    \item For each $J\in\cT_B$, $|Y_{B,k}(J)|\lesssim(\de/r)^{1-s-\eta}|J|$, and the one-dimensional $r^{-1}$-dilate of $Y_{B,k}(J)$ along $J$ is a $(\de/r, s, (\de/r)^{-\eta_1})$-set.
\end{enumerate}

\smallskip

We want to apply Proposition \ref{prop-induction} to the $r^{-1}$-dilate of $(\cT_B, Y_{B,k})$. 
However, the $r^{-1}$-dilate of $(\cT_B, Y_{B,k})$ may not obey its Assumption (iii).   
To get around this issue, we apply Lemma \ref{rich-point-refinement} to $(\cT_B, Y_{B,k})$ to obtain a refinement $(\cT_{B,k}', Y_{B,k}')$ of $(\cT_B,Y_{B,k})$ such that
\begin{enumerate}
    \item For any $J\in\cT_{B,k}'$, the one-dimensional $r^{-1}$-dilate of $Y_{B,k}'(J)$ along $J$ is a $(\de/r, s, (\de/r)^{-2\eta_1})$-set (achieved by Lemma \ref{refinement-delta-s-set}).
    \item Let $E_{\cT_{B,k}',Y_{B,k}'}=\bigcup_{T\in\cT_{B,k}'}Y_{B,k}'(J)$. For any $\de$-ball $p\subset E_{\cT_{B, k}', Y_{B, k}'}$, $\cT_{B,k}'(p):=\{J\in\cT_{B,k}':p\subset Y_{B,k}'(J) \}$ satisfies 
    \begin{equation}
        \#\cT_{B,k}'(p) \lessapprox|E_{\cT_{B,k}',Y_{B,k}'}|^{-1}\sum_{J\in\cT_{B,k}'}|Y_{B,k}'(J)|.
    \end{equation}
\end{enumerate}
Since $|Y_{B,k}(J)|$ are about the same for all $J\in\cT_B$, we have $\#\cT_{B,k}'\gtrapprox\#\cT_B$. 
Apply Lemma \ref{uniformization} to $\cT_{B,k}'$ so that there is a uniform refinement $\cT_{B,k}''$ of $\cT_{B,k}'$.
For each $J\in\cT_B$, let  $Y_{B,k}''(J)=Y_{B,k}'(J)$ if $J\in \cT_{B,k}''$;  otherwise, let $Y_{B,k}''(J)=\varnothing$. 

\smallskip

Recall that the parameter $\eta$ is chosen such that Proposition \ref{prop-induction} is true when $(\nu,\eta)=(\e^2,2\eta)$. 
Since $(\cT_{B,k}'', Y_{B,k}'')$ is a refinement of $(\cT_B, Y_{B,k})$ and since the uniform set $\cT_{B,k}''$ is also a refinement of the uniform set $\cT_B$, we know that
\begin{enumerate}
    \item For each $J\in\cT_{B,k}''$, the one-dimensional $r^{-1}$-dilate of $Y_{B,k}''(J)$ along the tube segment $J$ is a $(\de/r, s, (\de/r)^{-2\eta_1})$-set.
    \item For each $v\in[\de,r]$,  there are $\gtrsim (r/v)^{s}(\de/r)^{3\eta_1}$ distinct $v\times r$-tubes in $\cT_v(B)$ containing at least a tube $J\in\cT_{B,k}''$. 
    \item We have $\#\cT_{B,k}''(p)\leq \#\cT_{B,k}'(p)\lessapprox|E_{\cT_{B,k}',Y_{B,k}'}|^{-1}\sum_{J\in\cT_{B,k}'}|Y_{B,k}'(J)|\lessapprox\\|E_{\cT_{B,k}'',Y_{B,k}''}|^{-1}\sum_{J\in\cT_{B,k}''}|Y_{B,k}''(J)|$ for all $ p\subset E_{\cT_{B,k}'',Y_{B,k}''}$.
\end{enumerate}

Now we apply Proposition \ref{prop-induction} to the $r^{-1}$-dilate of $(\cT_{B,k}'', Y_{B,k}'')$.
After rescaling back, there is a refinement $(\cT_{B,k}''', Y_{B,k}''')$ of $(\cT_{B,k}'', Y_{B,k}'')$ and three possible outcomes:
\begin{enumerate}
    \item[{\bf A.}]  Define $\mu_{B,k}:=\max_{p\subset E_{\cT_{B,k}''', Y_{B,k}'''}} \# \cT_{B,k}'''(p) $, then $\mu_{B,k}\lessapprox (\de/r)^{(s-t)/2}$.
    \item[{\bf B.}] We have
    \begin{equation}
        |E_{\cT_{B,k}''', Y_{B,k}'''}|\gtrapprox (\de/r)^{\e^2+3\eta}(\de/r)^{2-(t+3s)/2}|B|.
    \end{equation}
    \item[{\bf C.}] 
     There exists a scale $\De_B\gtrsim r(\de/r)^{1-\sqrt{\eta}}$, and for each $J\in \cT_{B,k}'''$, two numbers $d_{B,1}(J), d_{B,2}(J)$ and two families of disjoint $\De_B$-balls $\cq_{B,1}(J), \cq_{B,2}(J)$ such that the following is true:
    \begin{enumerate}
        \item $\#\cq_{B,j}(J)= d_{B,j}(J), j=1,2$, and $\cup_{\cq_{B,2}(J)}\subset\cup_{\cq_{B,1}(J)}\subset N_{\De_B}(J)$.
        \item $|E_{\cT_{B,k}'', Y_{B,k}''}\cap Q|\gtrsim  (\de/r)^{1-s+2\eta}d_{B,1}(J)^{-1}(r/\De_B)|Q|$ for all $Q\in\cq_{B,1}(J)$.
        \item For each $Q\in\cq_{B,2}(J)$, $Q$ contains $\gtrsim (\de/r)^{-s-\eta}d_{B, 2}(J)^{-1}$ many $\de$-balls in $Y_{B,k}'''(J)$.
    \end{enumerate} 
\end{enumerate}

For each $k\leq K_B$, one of the above outcome happens for $L_{3,k}(B)$.  For $\bf X\in\{\bf A, B, C\}$, let $\ck_B({\bf X})$ be the set of $k$  such that outcome $\bf X$ holds for $L_{3,k}(B)$.  
Since $(\cT_{B,k}''', Y_{B,k}''')$ is a refinement of $(\cT_B, Y_{B,k})$ for each $k$, we know from \eqref{a-lot-of-J} that
\begin{equation}
    \int_{B}\#L_3(x)=\sum_{{k\leq K_B}}\sum_{p\subset E_{\cT_B, Y_{B,k}}}\#\cT_B(p)\lessapprox\sum_{k\leq K_B}\sum_{p\subset E_{\cT_{B,k}''', Y_{B,k}'''}}\#\cT_{B,k}'''(p).
\end{equation}
By pigeonholing, there exists an $\bf X\in\{\bf A, B, C\}$ such that 
\begin{equation}
\label{case-D}
    \int_{B}\#L_3(x)\lessapprox\sum_{k\in\ck_B(\bf X)}\sum_{p\subset E_{\cT_{B,k}''', Y_{B,k}'''}}\#\cT_{B,k}'''(p).
\end{equation}
We remark that \eqref{case-D} applies to each $B\in\cb_{r,5}$. 

\medskip

\noindent {\bf Step 4.2: Analyze $E_L\cap B$ for each $r$-ball $B$ in the broad case (II).}

Since $(L_5, Y_5)_\de$ is a refinement of $(L_3,Y_3)_\de$ and since for each $B\in\cb_{r,5}$, $B\cap Y_5(\ell)\not=\varnothing$ for some $\ell\in L_5$, we know that
\begin{equation}
    \int_{E_{L_3}}\# L_3(x)\lessapprox\int_{E_{L_5}}\# L_5(x)=\sum_{B\subset\in\cb_{r,5}}\int_{B}\#L_5(x)\leq\sum_{B\in\cb_{r,5}}\int_{B}\#L_3(x) .
\end{equation}
By \eqref{case-D} and pigeonholing, there exists a $\bf X\in\{\bf A, B, C\}$ and a subset $\cb_{r,5}({\bf X})\subset \cb_{r,5}$ such that
\begin{equation}
\label{Y-6-refinement}
    \int_{E_{L_3}}\# L_3(x)\lessapprox\sum_{B\in\cb_{r,5}(\bf X)}\sum_{k\in\ck_B(\bf X)}\sum_{p\subset E_{\cT_{B,k}''', Y_{B,k}'''}}\#\cT_{B,k}'''(p) .
\end{equation}

The configurations $\{(\cT_{B,k}''', Y_{B,k}'''):B\in\cb_{r,5}({\bf X}), k\in K_B(\bf X)\}$ define a new shading $Y_6$ for each $\ell\in L_3$ as follows: 
Let $\cb_{r,5}(\ell)\subset\cb_{r,5}({\bf X})$ be the family of $r$-balls obeying $Y_3(\ell)\cap B\not=\varnothing$. 
For each $B\in\cb_{r,5}(\ell)$, we know that 
\begin{enumerate}
    \item There is a $k\in\ck_B({\bf X})$ such that $\ell\in L_{3,k}(B)$.
    \item For this $k$, there is a $J\in\cT_B$ such that $Y_{B,k}(J)\supset Y_3(\ell)\cap B$.
\end{enumerate}
Now, for this $J\in\cT_B$,  we define
\begin{equation}
    Y_6(\ell)\cap B:=Y_{B,k}'''(J),\,\,\,\text{ and }\,\,\, Y_6(\ell)=\bigcup_{B\in\cb_{r,5}(\ell)}Y_6(\ell)\cap B.
\end{equation}
Let $L_6=L_3$. Then by \eqref{Y-6-refinement}, $(L_6, Y_6)_\de$ is a refinement of $(L_3, Y_3)_\de$.

\smallskip

The definition of  $(L_6, Y_6)_\de$ is subject to the outcome $\bf X\in\{A,B,C\}$ from pigeonholing.  
Let us discuss what happens for different possibilities of $\bf X$. Define $\cb_{r,6}=\cb_{r,5}(\bf X)$, so $E_{L_6}\subset \cup_{\cb_{r,6}}$.

\smallskip

When $\bf X=A$: For each $B\in\cb_{r,6}$, we have $\mu_{B,k} \lessapprox (\de/r)^{(s-t)/2}$ for each $k\in\ck_B(\bf X)$. As a result, for each $x\in E_{L_6} \cap B$,  since $\#\ck_B({\bf X})\lessapprox \mu_3$, 
\begin{equation}
    \# L_6(x)\leq \sum_{k\in\ck_B(\bf X)}\mu_{B,k}\lessapprox \#\ck_B({\bf X})\cdot \mu_{B,k}\lessapprox \mu_3\cdot (\de/r)^{(s-t)/2}.
\end{equation}

\smallskip

When $\bf X=B$: For each $B\in\cb_{r,6}$, since $E_{\cT_{B,k}''', Y_{B,k}'''}\subset E_{L_6}\cap B$ for any $k\in\ck_B(\bf X)$, 
\begin{equation}
    |E_{L_6}\cap B|\geq\max_{k\in\ck_B(\bf X)}|E_{\cT_{B,k}''', Y_{B,k}'''}|\gtrapprox (\de/r)^{\e^2+3\eta}(\de/r)^{2-(t+3s)/2}|B|.
\end{equation} 

\smallskip

When $\bf X=C$: For each $B\in\cb_{r,6}$ and each  $k\in\ck_B(\bf X)$, there exists a scale $\De_B\gtrsim r(\de/r)^{1-\sqrt{\eta}}$, and for each $J\in \cT_{B,k}'''$, two numbers $d_{B,1}(J), d_{B,2}(J)$ and two families of disjoint $\De_B$-balls $\cq_{B,1}(J), \cq_{B,2}(J)$ such that the following is true:
\begin{enumerate}
    \item$\#\cq_{B,j}(J)= d_{B,j}(J), j=1,2$, and $\cup_{\cq_{B,2}(J)}\subset \cup_{\cq_{B,1}(J)}\subset N_{\De_B}(J)$.
    \item As $E_{\cT_{B,k}'', Y_{B,k}''}\subset E_{L_3}$, we have $|E_{L_3}\cap Q|\gtrsim  (\de/r)^{1-s+2\eta}d_{B,1}(J)^{-1}(r/\De_B)|Q|$ for each $Q\in\cq_{B,1}(J)$.
    \item For each $Q\in\cq_{B,2}(J)$, $Q$ contains $\gtrsim (\de/r)^{-s-\eta}d_{B, 2}(J)^{-1}$ many $\de$-balls in $Y_{B,k}'''(J)$.
\end{enumerate}

\smallskip

Let us summarize what we have so far: $(L_6, Y_6)_\de$ is a refinement of $(L_2, Y_2)_\de$, and for $(L_6, Y_6)_\de$, either one of the following must happen:
\begin{enumerate}
    \item[{\bf A.}]  For each $x\in E_{L_6}$, 
    \begin{equation}
        \# L_6(x)\leq \sum_{k\in\ck_B(\bf X)}\mu_{B,k}\lessapprox \#\ck_B({\bf X})\cdot \mu_{B,k}\lessapprox \mu_3\cdot (\de/r)^{(s-t)/2}.
    \end{equation}
    \item[{\bf B.}] For each $B\in\cb_{r,6}$, we have
    \begin{equation}
        |E_{L_6}\cap B|\geq\max_{k\in\ck_B(\bf X)}|E_{\cT_{B,k}''', Y_{B,k}'''}|\gtrapprox (\de/r)^{\e^2+3\eta}(\de/r)^{2-(t+3s)/2}|B|.
    \end{equation} 
    \item[{\bf C.}] For all $B\in\cb_{r,6}$, there exists a scale $\De_B\geq r(\de/r)^{1-\sqrt{\eta}}$ such that for all $\ell\in L_6$, there are two numbers $d_{B,1}(\ell), d_{B,2}(\ell)$ and two families of disjoint $\De_B$-balls $\cq_{B,1}(\ell), \cq_{B,2}(\ell)$, and the following is true:
    \begin{enumerate}
        \item $\#\cq_{B,j}(\ell)= d_{B,j}(\ell), j=1,2$, and $\cup_{\cq_{B,2}(\ell)}\subset\cup_{\cq_{B,1}(\ell)}\subset N_{\De_B}(\ell)\cap B$.
        \item $|E_{L_3}\cap Q|\gtrsim  (\de/r)^{1-s+2\eta}d_{B,1}(\ell)^{-1}(r/\De_B)|Q|$ for each $Q\in\cq_{B,1}(\ell)$.
        \item For each $Q\in\cq_{B,2}(\ell)$, $Q$ contains $\gtrsim (\de/r)^{-s-\eta}d_{B, 2}(\ell)^{-1}$ many $\de$-balls in $Y_6(\ell)\cap B$.
    \end{enumerate} 
\end{enumerate}
We proceed to the next step.

\medskip

\noindent {\bf Step 5: Estimate $|E_L|$: Three  cases.}

\smallskip  

{\bf Case A.} Suppose item {\bf A} happens.
By \eqref{small-multi-assumption} and \eqref{tilde-lambda}, we get 
\begin{align}
    \#L_6(x)&\lessapprox c_\e^{-1} r^{-\e}(C\tilde C_2)^{\eta^{-2}} \tilde\la^{-1/2} r^{(1-t)/2}\ga_{ Y_2, t^\ast}^{1/2}\cdot (\de/r)^{(s-t)/2}\\
    &\lessapprox c_\e^{-1}r^{-\e}(r/\de)^{\eta_1}(C\tilde C_2)^{\eta^{-2}}\la^{-1/2}\de^{(1-t)/2}\ga_{Y_2, t^\ast}^{1/2}.
\end{align}
Since $(L_6, Y_6)_\de$ is a refinement of $(L,Y)_\de$ and since $\ga_{Y_2, t^\ast}\leq \ga_{Y, t^\ast}$, we thus have,
\begin{align}
    |E_L|\geq|E_{L_6}|&\gtrapprox c_\e r^{\e}(r/\de)^{-\eta_1}(C\tilde C_2)^{-\eta^{-2}} \la^{1/2}\de^{(t-1)/2}\ga_{Y,t^\ast}^{-1/2}\sum_{\ell\in L_6}|Y_6(\ell)|\\
    &\geq c_\e\de^{\e}(C\tilde C_2)^{-\eta^{-2}}\la^{1/2}\de^{(t-1)/2}\ga_{Y,t^\ast}^{-1/2}\sum_{\ell\in L}|Y(\ell)|.
\end{align}
The last inequality follows from the fact $\eta_1\leq\e$.
This concludes \eqref{main-prop-esti}.

\medskip 

{\bf Case B.}  Suppose item {\bf B} happens. Then for each $B\in\cb_{r,6}$,
\begin{equation}
\label{high-density-in-r-ball}
    |E_{L}\cap B|\geq |E_{L_6}\cap B|\gtrapprox (\de/r)^{\e^2+3\eta}(\de/r)^{2-(t+3s)/2}|B|.
\end{equation}
Let $L_7$ be the set of lines $\ell\in L_6$ such that $|Y_6(\ell)| \gtrapprox |Y_2(\ell)|\approx\la\delta$.
Since $(L_6,Y_6)_\de$ is a refinement of $(L_2,Y_2)_\de$,
$\#L_7 \gtrapprox \#L_2$. 
For each $\ell\in L_7$, let  $Y_7(\ell)$ be a uniform refinement of $ Y_6(\ell)$, so $|Y_7(\ell)|\approx \la\delta$.
Note that $(L_7,Y_7)_\de$ is a refinement of $(L_2,Y_2)_\de$.

For each $\ell\in L_7$, let $\cb_{r,6}(\ell)\subset\cb_{r,6}$ be so that $Y_7(\ell)\cap B\not=\varnothing$ for all $B\in\cb_{r,6}(\ell)$.
Now define a new shading 
\begin{equation}
    \tilde Y_8(\ell)=\cup_{\cb_{r,6}(\ell)}\cap N_{r}(\ell)
\end{equation}
by $r$-balls. 
Since $|Y_7(\ell)\cap B|\lesssim (\de/r)^{1-s-\eta_1}|N_\de(\ell)\cap B|$ for each $B\in \cb_{r, 6}(\ell)$ (from the setup before Step 1), we have (recall \eqref{tilde-lambda})
\begin{equation}
  |\tilde Y_8(\ell)|/ |N_r(\ell)| \gtrsim \de^{-1}|Y_7(\ell)|\cdot (r/\de)^{1-s-\eta_1}\gtrapprox \tilde \la.
\end{equation}
Moreover, since $\ga_{\tilde Y_2, t^\ast}\lesssim \ga_{Y_2, t^\ast}$, we have $\ga_{\tilde Y_8, t^\ast}\lesssim \ga_{Y_2, t^\ast}$

\smallskip

By Lemma \ref{Katz--Tao-set-lem}, there exists $\tilde L_8\subset L_7$ such that $(r^t\# \tilde L_8)\gtrapprox (\de^t \# L_7)$, and $\tilde L_8$ is a Katz--Tao $(r,t)$-set. 
For all $\ell\in L_7$, since $Y_7(\ell)$ is a refinement of $Y_2(\ell)$ and since $Y_2(\ell)$ is a $(\de,\e^2,CC_2;\rho)$-set for some $C_2\lessapprox1$, $Y_7(\ell)$ is a $(\de,\e^2,CC_7;\rho)$-set for some $C_7\lessapprox1$.
Since $Y_7(\ell)$ is uniform, $\tilde Y_8(\ell)$ is a $(r,\e^2,CC_8;\max\{\rho,r\})$-set for some $C_8\lessapprox1$.
Clearly, $\max\{\rho,r\}\in[r,r^\eta]$, since $\rho\in[\de,\de^\eta]$.
Apply Proposition \ref{two-ends-furstenberg-prop} at scale $r$ to $(\tilde L_8, \tilde Y_8)_{r}$ so that 
\begin{align}
    |E_{\tilde L_8}|&\geq c_\e r^{\e}(C\tilde C_8)^{-\eta^{-2}} \tilde{\lambda}^{1/2}r^{(t-1)/2}\ga_{\tilde Y_8, t^\ast}^{-1/2} \sum_{\ell\in \tilde L_8}|\tilde Y_8(\ell)|\\
    & \gtrsim c_\e r^{\e}(C\tilde C_8)^{-\eta^{-2}} \tilde{\lambda}^{3/2}r^{(t-1)/2}\ga_{\tilde Y_8, t^\ast}^{-1/2} (r \# \tilde L_8).
\end{align}
Since $(r^t\# \tilde L_8)\gtrapprox (\de^t \# L_7)$, since $\ga_{\tilde Y_8, t^\ast}\lesssim \ga_{Y_2, t^\ast}$, and since $|Y_7(\ell)|\approx \la\delta$ for all $\ell\in L_7$, the above gives
\begin{align}
\nonumber
    |E_{\tilde L_8}|& \gtrsim c_\e r^{\e}(C\tilde C_8)^{-\eta^{-2}} \tilde{\lambda}^{3/2}r^{(t-1)/2}\ga_{Y_2, t^\ast}^{-1/2}\cdot   \lambda^{-1}(r/\de)^{1-t} \sum_{\ell \in L_7} |Y_7(\ell)|\\ \label{E-tilde-L8}
    &\gtrsim c_\e r^{\e}(C\tilde C_8)^{-\eta^{-2}}\la^{1/2}r^{(t-1)/2}\ga_{Y_2, t^\ast}^{-1/2}(r/\de)^{3(1-s-\eta_1)/2}(r/\de)^{1-t}\sum_{\ell\in L_7}|Y_7(\ell)|.
\end{align}
Therefore, by \eqref{high-density-in-r-ball}, \eqref{E-tilde-L8} implies
\begin{align}
    |E_L|& \gtrapprox  \min_{B\in\cb_{r,6}}\frac{  |E_L\cap B| }{|B|} \cdot |E_{\tilde L_8}| \gtrapprox (\de/r)^{\e^2+3\eta}(\de/r)^{2-(t+3s)/2}|E_{\tilde L_8}|\\[2ex]
    &\gtrapprox (\de/r)^{\e^2+3\eta+3\eta_1/2}\cdot c_\e r^{\e}(C\tilde C_8)^{-\eta^{-2}}\la^{1/2}\de^{(t-1)/2}\ga_{Y_2, t^\ast}^{-1/2}\sum_{\ell\in L_7}|Y_7(\ell)|\\
    &\gtrapprox (\de/r)^{\e^2+3\eta+2\eta_1-\e}\cdot c_\e \de^{\e}(C\tilde C_8)^{-\eta^{-2}}\la^{1/2}\de^{(t-1)/2}\ga_{Y, t^\ast}^{-1/2}\sum_{\ell\in L}|Y(\ell)|.
\end{align}
Since $\eta_1\leq\eta\leq\e^2$, we have  $(\de/r)^{\e^2+3\eta+2\eta_1-\e}\geq1$. 
This proves \eqref{main-prop-esti}.

\medskip

{\bf Case C.} Suppose item {\bf C} happens. Then for each $B\in\cb_{r,6}$,  there exists a scale $\De_B\geq r(\de/r)^{1-\sqrt{\eta}}$.
Note that $Y_6(\ell)\subset\cup_{\cb_{r,6}}$for any $\ell\in L_6$. Since $(L_6,Y_6)_\de$ is a refinement of $(L_2,Y_2)_\de$, we have
\begin{equation}
    \int_{E_{L_2}}\# L_2(x)\lessapprox\int_{\cup_{\cb_{r,6}}}\# L_6(x).
\end{equation}
By dyadic pigeonholing, there exist  $\cb_{r,6}'\subset\cb_{r,6}$ and $\De\geq r(\de/r)^{1-\sqrt{\eta}}$ so that

\begin{enumerate}
    \item $ \int_{E_{L_2}}\# L_2(x)\lessapprox \int_{\cup_{\cb_{r,6}'}}\# L_6(x)$.
    \item For all $B\in\cb_{r,6}'$ and all $\ell\in L_6$, there are two numbers $d_{B,1}(\ell), d_{B,2}(\ell)$ and two families of disjoint $\De$-balls $\cq_{B,1}(\ell), \cq_{B,2}(\ell)$ such that:
    \begin{enumerate}
        \item $\#\cq_{B,j}(\ell)= d_{B,j}(\ell), j=1,2$, and $\cup_{\cq_{B,2}(\ell)}\subset\cup_{\cq_{B,1}(\ell)}\subset N_{\De}(\ell)\cap B$.
        \item $|E_{L_3}\cap Q|\gtrsim  (\de/r)^{1-s+2\eta}d_{B,1}(\ell)^{-1}(r/\De)|Q|$ for each $Q\in\cq_{B,1}(\ell)$.
        \item For each $Q\in\cq_{B,2}(\ell)$, $Q$ contains $\gtrsim (\de/r)^{-s-\eta}d_{B, 2}(\ell)^{-1}$ many $\de$-balls in $Y_6(\ell)\cap B$.
    \end{enumerate} 
\end{enumerate}
Thus, if denoting $Y_6'(\ell)=Y_6(\ell)\cap\cup_{\cb_{r,6}'}$ for all $\ell\in L_6$, then $(L_6,Y_6')_\de$ is a refinement of $(L_2,Y_2)_\de$. 
Let $(L_7,Y_7)_\de$ be a refinement of $(L_6,Y_6')_\de$ such that $|Y_7(\ell)|\gtrapprox |Y_2(\ell)|$ and $Y_7(\ell)$ is uniform for all $\ell\in L_7$. 
Hence $|Y_7(\ell)|\approx \la\de$ for all $\ell\in L_7$.

\smallskip

For each $\ell\in L_7$, let $\cb_{r,6}'(\ell)\subset\cb_{r,6}'$ be so that $Y_7(\ell)\cap B\not=\varnothing$ for any $B\in\cb_{r,6}'(\ell)$.
By dyadic pigeonholing on $\{d_{B,1}(\ell), d_{B,2}(\ell):B\in\cb_{r,6}'(\ell)\}$, there is a set $\cb_{r,6}''(\ell)\subset \cb_{r,6}'(\ell)$ and two numbers $d_1(\ell),d_2(\ell)$ such that
\begin{enumerate}
    \item For all $B\in\cb_{6,r}''(\ell)$, $d_{B,j}(\ell)\sim d_j(\ell), j=1,2$.
    \item $|Y_7'(\ell)|\gtrapprox|Y_7(\ell)|\approx\la\de$, where $Y_7'(\ell):=Y_7(\ell)\cap\cup_{\cb_{6,r}''(\ell)}$.
\end{enumerate}
Thus, $(L_7, Y_7')_\de$ is a refinement of $(L_6,Y_6)_\de$.
By dyadic pigeonholing again, there is a refinement $L_8$ of $L_7$ and two uniform numbers $d_1, d_2$ such that for all $\ell\in L_8$, $Y_7'(\ell)$ is $\gtrapprox1$-refinement of $Y_7(\ell)$, and $d_j(\ell)\sim d_j,j=1,2$.
The configuration $(L_8, Y_7', d_1, d_2)$ satisfies that for all $\ell\in L_8$, all $B\in\cb_{6,r}''(\ell)$, we have the following:
\begin{enumerate}
    \item  $\cup_{\cq_{B,2}(\ell)}\subset\cup_{\cq_{B,1}(\ell)}\subset N_{\De}(\ell)\cap B$.
    \item $|E_{L_3}\cap Q|\gtrsim  (\de/r)^{1-s+2\eta}d_1^{-1}(r/\De)|Q|$ for each $Q\in\cq_{B,1}(\ell)$.
    \item For each $Q\in\cq_{B,2}(\ell)$, $Q$ contains $\gtrsim (\de/r)^{-s-\eta}d_{2}^{-1}$ many $\de$-balls in $Y_7(\ell)\cap B$.
\end{enumerate}

\smallskip

Now for each $\ell\in L_8$, we define three new shadings 
\begin{equation}
\label{caseC-Y-8}
    Y_8(\ell)=Y_7'(\ell)\bigcap\cup_{\cb_{r,6}'''(\ell)},\,\,\,\tilde Y_8(\ell)=\bigcup_{B\in\cb_{r,6}'''(\ell)}\cup_{\cq_{B,1}(\ell)}, \,\,\,\tilde Y_8'(\ell)=\bigcup_{ B\in\cb_{r,6}''(\ell)}\cup_{\cq_{B,2}(\ell)}.
\end{equation}
Since $d_1|\tilde Y_8'(\ell)|\sim d_2|\tilde Y_8(\ell)|$, we have $\ga_{\tilde Y_8, t^\ast}\lesssim (d_1/d_2)\ga_{\tilde Y_8', t^\ast}$.
Also, by Lemma \ref{gamma-y-lemma}, we know $\ga_{\tilde Y_8', t^\ast}\lesssim(\de/\De)^{-t^\ast}(\de/r)^{s+\eta}d_2\ga_{Y_7, t^\ast}$. 
Thus, we have
\begin{equation}
    \ga_{\tilde Y_8, t^\ast}\lesssim(\de/\De)^{-t^\ast}(\de/r)^{s+\eta}d_1\ga_{Y_7, t^\ast}.
\end{equation}

Define $d:=(\de/r)^{-s-\eta}(\de/\De)d_1^{-1}$ to simplify notation. 
Then 
\begin{equation}
\label{gamma-8-gamma-7}
    \ga_{\tilde Y_8, t^\ast}\lesssim(\de/\De)^{1-t^\ast}d^{-1}\ga_{Y_7, t^\ast}.
\end{equation}
Moreover, for all $\ell\in L_8$, all $B\in\cb_{6,r}''(\ell)$, and all $Q\in\cq_{B,1}(\ell)$, we have 
\begin{equation}
\label{density-Delta-ball}
    |E_{L_3}\cap Q|\gtrsim  (\de/r)^{3\eta}d|Q|.
\end{equation}

Since for each $r$-ball $B\in \cb_{r,6}'$, 
$|Y_7(\ell)\cap B|\lesssim|Y_2(\ell)\cap B|\lesssim (\de/r)^{1-s-\eta_1}|N_\de(\ell)\cap B|$, and since $|Y_7(\ell)|\approx|Y_7'(\ell)|\approx \la\de$,
we know that $\#\cb_{r,6}''(\ell)\gtrapprox (\de r)^{-1}|Y_7(\ell)| (r/\de)^{1-s-\eta_1}$.
Therefore, $\tilde Y_8(\ell)$ contains $\gtrapprox d_1\cdot(\de r)^{-1}|Y_7(\ell)| (r/\de)^{1-s-\eta_1}=(\de/r)^{-s-\eta}(\de/\De)d^{-1}\cdot (\de r)^{-1}\la\de (r/\de)^{1-s-\eta_1}$ many $\De$-balls. 
Simplify the equation to get that $\tilde Y_8(\ell)$ contains $\gtrapprox(\de/r)^{\eta_1-\eta}d^{-1}\De^{-1}\la$ many $\De$-balls. 
This gives
\begin{align}
    |\tilde Y_8(\ell)|&\gtrapprox (\de/r)^{\eta_1-\eta}d^{-1}\De^{-1}\la\cdot\De^2\sim (\de/r)^{\eta_1-\eta}d^{-1}\la\De.
\end{align}
In particular, $\tilde Y_8$ is $\gtrapprox (\de/r)^{\eta_1-\eta}d^{-1}\la$-dense.

\smallskip

By Lemma \ref{Katz--Tao-set-lem}, there exists $\tilde L_8\subset L_8$ such that $(\De^t\# \tilde L_8)\gtrapprox (\de^t \# L_8)$, and $\tilde L_8$ is a Katz--Tao $(\De,t)$-set.
Since $L_8$ is a refinement of $L_7$, $(\De^t\# \tilde L_8)\gtrapprox (\de^t \# L_7)$.
Now we are going to use the assumption $r\leq\de^{\eta}$.
For all $\ell\in L_8$, since $Y_8(\ell)$ is a refinement of $Y_2(\ell)$ and since $Y_2(\ell)$ is a $(\de,\e^2,CC_2;\rho)$-set for some $C_2\lessapprox1$, $Y_8(\ell)$ is a $(\de,\e^2,CC_8;\rho)$-set for some $C_8\lessapprox1$.
Since $Y_7(\ell)$ is uniform, by the definition of $Y_8$ in \eqref{caseC-Y-8}, $(Y_8(\ell))_r$ is a $(r,\e^2,C\tilde C_8;\max\{\rho,r\})$-set for some $\tilde C_8\lessapprox1$.
Finally, since $|\tilde Y_8(\ell)\cap J|_\De\sim d$ for all $J\subset (\tilde Y_8(\ell))_r$ and by \eqref{caseC-Y-8}, 
$\tilde Y_8(\ell)$ is a $(\De,\e^2,C\tilde C_8;\max\{\rho,r\})$-set.
Note that $\max\{\rho,r\}\in[\De,\De^\eta]$, since $\rho\in[\de,\de^\eta]$ and since $r\leq\de^{\eta}$.
Apply induction at scale $\De$ to $(\tilde L_8, \tilde Y_8)_{\De}$ so that 
\begin{align}
\nonumber
    &|E_{\tilde L_8}| \gtrapprox   c_\e \De^{\e}(C\tilde C_8)^{-\eta^{-2}}(\de/r)^{(\eta_1-\eta)/2}d^{-1/2}\la^{1/2}\De^{(t-1)/2}\ga_{\tilde Y_8, t^\ast}^{-1/2}\sum_{\ell\in \tilde L_8}|\tilde Y_8(\ell)|\\ \nonumber
    &\gtrapprox  c_\e \De^{\e}(C\tilde C_8)^{-\eta^{-2}}(\de/r)^{3(\eta_1-\eta)/2}d^{-1/2}\la^{1/2}\De^{(t-1)/2}\ga_{\tilde Y_8, t^\ast}^{-1/2}d^{-1}(\de/\De)^{t-1}\sum_{\ell\in L_7}|Y_7(\ell)|\\ \nonumber
    &\gtrapprox  c_\e \De^{\e}(C\tilde C_8)^{-\eta^{-2}}(\de/r)^{3(\eta_1-\eta)/2}\la^{1/2}\de^{(t-1)/2}\ga_{Y_7, t^\ast}^{-1/2}d^{-1}\sum_{\ell\in L_7}|Y_7(\ell)|.
\end{align}
In the last inequality, we  use \eqref{gamma-8-gamma-7} (note that when $t<1$, $(\de/\De)^{1-t^\ast}=(\de/\De)^{1-t}$).

Recall \eqref{density-Delta-ball} that for all $\De$-ball $Q\subset E_{\tilde L_8}$, $|E_{L_3}\cap Q|\gtrsim d(\de/r)^{3\eta}|Q|$. 
Thus, since $(L_7, Y_7)_\de$ is a refinement of $(L,Y)_\de$, since $\ga_{Y_7, t^\ast}\leq\ga_{Y, t^\ast}$, and since $\eta_1\leq\eta$, we get
\begin{align}
    |E_L|&\geq d(\de/r)^{3\eta}|E_{\tilde L_8}|\gtrapprox c_\e \De^{\e}(C\tilde C_8)^{-\eta^{-2}}(\de/r)^{3\eta}\la^{1/2}\de^{(t-1)/2}\ga_{Y_7, t^\ast}^{-1/2}\sum_{\ell\in L}|Y(\ell)|\\
    &  \gtrapprox (\de/r)^{3\eta}(\De/\de)^{\e}\cdot c_\e \de^{\e}(C\tilde C_8)^{-\eta^{-2}}\la^{1/2}\de^{(t-1)/2}\ga_{Y, t^\ast}^{-1/2}\sum_{\ell\in L}|Y(\ell)|.
\end{align}
Since $\De\geq r(\de/r)^{1-\sqrt{\eta}}$, we have $\De/\de\geq (r/\de)^{\sqrt{\eta}}$.
Also, since $\eta_1\leq\eta\leq\e^{10}$, we get $(\de/r)^{3\eta}(\De/\de)^{-\e}\geq (r/\de)^{\e\sqrt{\eta}-3\eta}\geq (r/\de)^{\eta}\geq1$. 
Consequently, we have $|E_L|\geq  c_\e \de^{\e}(C\tilde C_8)^{-\eta^{-2}}\la^{1/2}\de^{(t-1)/2}\ga_{Y, t^\ast}^{-1/2}\sum_{\ell\in L}|Y(\ell)|$.
This concludes Proposition \ref{two-ends-furstenberg-prop}. \qedhere

\end{proof}

\bigskip

\bibliographystyle{alpha}
\bibliography{bibli}

\end{document}